\documentclass[11pt,a4paper,twoside]{article}
\usepackage{amssymb}
\usepackage{amsfonts}
\usepackage{geometry}
\usepackage{color}
\usepackage{amsmath,amsthm,amssymb,amsfonts}
\usepackage{CJK}
\usepackage{latexsym}
\usepackage{graphicx}
\usepackage{pgfpages}
\usepackage{mathrsfs}
\usepackage{ytableau}
\usepackage{enumerate}
\usepackage[affil-it]{authblk}
\usepackage{xcolor}
\allowdisplaybreaks

\newtheorem{definition}{Definition}[section]
\newtheorem{theorem}[definition]{Theorem}
\newtheorem{lemma}[definition]{Lemma}
\newtheorem{proposition}[definition]{Proposition}

\newtheorem{remark}[definition]{Remark}
\newtheorem{condition}[definition]{Condition}
\newtheorem{example}[definition]{Example}

\numberwithin{equation}{section}

\allowdisplaybreaks

\begin{document}
	\title{Random quasi-periodic paths and quasi-periodic measures of stochastic differential equations}

	\author[*]{Chunrong Feng}
	\author [$\dagger$,*] {Baoyou Qu}
	\author[*,$\ddagger$]{Huaizhong Zhao}
	\affil[*]
	{Department of Mathematical Sciences, Durham
		University, DH1 3LE, UK}
	\affil[$\dagger$] {Zhongtai Securities Institute for Financial Studies, Shandong University, Jinan 250100, China}
	\affil[$\ddagger$]{Research Centre for Mathematics and Interdisciplinary Sciences, Shandong University, Qingdao 266237, China}
	\affil[ ]{chunrong.feng@durham.ac.uk, qu@mail.sdu.edu.cn, huaizhong.zhao@durham.ac.uk}
	\date{}
	
\maketitle

\begin{abstract}
	 In this paper, we define random quasi-periodic paths for random dynamical systems and quasi-periodic
	 measures for Markovian semigroups. We give a sufficient condition for the existence and uniqueness of random quasi-periodic 
	paths and quasi-periodic measures for stochastic differential equations and a sufficient condition for the density of the quasi-periodic measure
	to exist and to satisfy the Fokker-Planck equation. We obtain an invariant measure by considering lifted
	flow and semigroup on cylinder and the tightness of the average of lifted quasi-periodic measures. We further prove that the invariant measure 
	is unique, and thus ergodic.
	
	\medskip
	
	
\end{abstract}

\section{Introduction}
Quasi-periodic oscillation of a dynamical system is a motion given by a quasi-periodic function $F$ such that
\begin{equation}
	\label{quasi-periodic function}
	F(t)=f(t,t,\cdots,t),
\end{equation}
  for some continuous function $f(t_1,t_2,\cdots,t_m), \ (t_1,t_2,\cdots,t_m)\in {\mathbb R}^m $ $(m\geq 2)$ which is periodic in $t_1,t_2,\cdots,t_m$ with periods $\tau_1,\tau_2,\cdots,\tau_m$ respectively, where $\tau_1,\tau_2,\cdots,\tau_m$ are strictly positive and their reciprocals are rationally linearly independent i.e. for any nonzero integer-valued vector $k=(k_1,k_2,\cdots,k_m)$,
  $$k_1\frac{1}{\tau_1}+k_2\frac{1}{\tau_2}+\cdots+k_m\frac{1}{\tau_m}\neq 0.$$
  This topic has been subject to many important studies including 
  Kolmogorov-Arnold-Moser (KAM) theory on Hamiltonian systems (\cite{Kolmogorov1954},\cite{Moser1962},\cite{Arnold1963}).
  
  Quasi-periodic motion is a common phenomenon in nature, e.g. arising in describing the movement of planets 
  around the sun. The existence of a quasi-periodic motion 
for the nearly integrable regimes of the three-body problem with
some transversality condition is given by the KAM theory. 
  However many problems in nature are mixture of randomness and quasi-periodic motions. 
  For example the temperature process which is random has one year periodicity due to the revolution of the earth around 
  the sun and one day-night periodicity due to the rotation of the earth. Similarly, the energy demands should have similar nature.
Thus to provide a rigorous mathematical theory is key in modelling random quasi-periodic phenomena in real world. As far as 
we know, such a concept still does not exist and the current paper is the first attempt in this direction.

  The concepts of random periodic paths and periodic measures were introduced recently (\cite{Zhao-Zheng2009},\cite{Feng-Zhao-Zhou2011},\cite{Feng-Zhao2012},\cite{Feng-Wu-Zhao2016},\cite{Feng-Zhao2018}). They are two different indispensable ways to describe random periodicity. The theory has led to progress in the study of bifurcations (\cite{Wang2014}), random attractors (\cite{Bates-Lu-Wang2014}), stochastic resonance (\cite{Cherubini-Lamb-Rasmussen-Sato2017},\cite{FZZ19}), strange attractors (\cite{Huang-Lian2016}) and modelling the El N\^ino phenomenon (\cite{Chekroun-Simonnet-Ghil2011}).
  
  In this paper, we study random quasi-periodicity of random dynamical systems or semi-flows over a metric dynamical system $(\Omega,\mathcal{F},P, (\theta_t)_{t\in\mathbb{R}})$. First we define random quasi-periodic path $\varphi$ of the stochastic-flows $u(t,s): \Omega\times\mathbb{R}^d\rightarrow \mathbb{R}^d, t\geq s$ as a random path satisfying
  $$u(t,s,\varphi(s,\omega),\omega)=\varphi(t,\omega),  t\geq s,s\in \mathbb{R} \text{ a.s.},$$
  and the pull-back random path
  $$t\longmapsto \varphi(t,\theta_{-t}\omega)$$
  is a quasi-periodic function for almost every sample path $\omega\in\Omega$.
  
  For a Markovian semi-flow, let $p(t,s,x,\cdot), t\geq s,$ be its transition probability. Then a measure-valued function $\rho: \mathbb{R}\rightarrow \mathcal{P}(\mathbb{R}^d)$ is called a quasi-periodic measure if $\rho$ is an entrance measure i.e.
  $$\int_{\mathbb{R}^d}P(t,s,x,\Gamma)\rho_s(dx)=\rho_t(\Gamma)\ \ {\rm \ for \ all}\  \Gamma \in {\cal B}({\mathbb R}^d),$$
  and the measure-valued map
  $$s\longmapsto \rho_s$$
  is a quasi-periodic function.
  
  We will give a sufficient condition for the existence and uniqueness of random quasi-periodic path for a stochastic differential equation on $\mathbb{R}^d$
  \begin{equation}
  \label{SDE}
  \begin{cases}
  dX(t)=b(t, X(t))dt+\sigma(t, X(t))dW_t,  \quad t\geq s,\\
  X(s)=\xi,
  \end{cases}
  \end{equation}
  where $b,\sigma$ are quasi-periodic in the time variable t. As this is the first paper in this area, the main 
  purpose here is to establish basic mathematical concepts and useful tools. We do not strike to 
  technical details to try to provide best possible sufficient conditions in the current paper. 
   
  We will prove the law of random quasi-periodic path is a quasi-periodic measure. We further give a sufficient condition for the density of the quasi-periodic measure to exist and to satisfy the Fokker-Planck equation.
  
  For simplicity, we only consider quasi-periodicity with two periods: $\tau_1$ and $\tau_2$ in the current paper. Our results also apply to general cases with any periods $\tau_1,\tau_2,\cdots,\tau_m$ without any extra difficulties.
  
  Solving the reparameterised SDE is a key step in the analysis of finding random quasi-periodic paths. Let $\tilde{b}, \tilde{\sigma}$ be two functions such that
  $$\tilde{b}(t,t,x)=b(t,x), \tilde{\sigma}(t,t,x)=\sigma(t,x)$$
  where $\tilde{b}(t_1,t_2,x), \tilde{\sigma}(t_1,t_2,x)$ are periodic in $t_1,t_2$ with periods $\tau_1$ and $\tau_2$ respectively. Define
  $$\tilde{b}^{r_1,r_2}(t,x)=\tilde{b}(t+r_1,t+r_2,x)$$ $$\tilde{\sigma}^{r_1,r_2}(t,x)=\tilde{\sigma}(t+r_1,t+r_2,x),$$
  then the solution $K^{r_1,r_2}$ of SDE (\ref{SDE}) when $b,\sigma$ are replaced by $\tilde{b}^{r_1,r_2}, \tilde{\sigma}^{r_1,r_2}$, where $r_1,r_2$ are regarded as parameters, satisfies
  $$K^{r,r}(t,s,x,\omega)=u(t+r,s+r,x,\theta_{-r}\omega)$$
  where $u(t,s,\cdot,\omega)$ is the semi-flow generated by (\ref{SDE}). Moreover we can prove under a dissipative condition about the drifts $b$ and $\tilde{b}^{r_1,r_2}$,
  $$\lim_{s\rightarrow -\infty}K^{r_1,r_2}(t,s,x,\omega)={\varphi}^{r_1,r_2}(t,\omega) \text{ exists a.s.}$$
  and
  $$\varphi(r,\omega)={\varphi}^{r,r}(0,\theta_{-r}\omega)$$
  is a random quasi-periodic path of (\ref{SDE}).

Note the reparamerterised SDE enjoys the following property: for all $r_1,r_2,r\in \mathbb{R}$, $t\geq s$,
	\begin{eqnarray}\label{2019aug3}
	K^{r_1,r_2}(t+r,s+r,x,\theta_{-r}\omega)=K^{r_1+r,r_2+r}(t,s,x,\omega), \ P-a.s. \text{ on }\omega.
	\end{eqnarray}
	This is a very useful observation in our analysis, but the original time dependent SDE (\ref{SDE}) does not have
	such a convenient relation.

  Lifting the semi-flow to $\tilde{\mathbb X}=[0, \tau_1) \times [0, \tau_2)\times\mathbb{R}^d$ is key to obtain an invariant measure from the quasi-periodic measure. Define
  $$\tilde{\Phi}(t,\omega)(s_1,s_2,x)=(t+s_1 \mod\tau_1,\ t+s_2 \mod\tau_2,\ K^{s_1,s_2}(t,0,x,\omega))$$
  and
  $$\tilde{Y}(s,\omega)=(s\mod \tau_1,\ s\mod\tau_2,\ \varphi(s,\omega)).$$
  Then $\tilde{Y}$ is a random quasi-periodic path of the cocycle $\tilde{\Phi}$. Moreover we will prove that $\tilde{P}(t,(s_1,s_2,x), \tilde{\Gamma})=P\{\omega: \tilde{\Phi}(t,\omega)(s_1,s_2,x)\in \tilde{\Gamma}\}, \tilde{\Gamma}\in\mathcal{B}(\tilde{\mathbb X})$ is Feller and 
  $$\tilde{\mu}_s(\tilde{\Gamma})=P\{\omega:\tilde{Y}(s,\omega)\in\tilde{\Gamma}\}=[\delta_{s\mod \tau_1}\times \delta_{s\mod \tau_2}\times \rho_s](\tilde{\Gamma})$$
 is a quasi-periodic measure with respect to $\tilde{P}^*$. We will show that 
 $$\{\bar{\tilde{\mu}}_T=\frac{1}{T}\int_{0}^{T}\tilde{\mu}_sds: T\in \mathbb{R}^+\}$$
 is tight and a weak limit $\bar{\tilde{\mu}}$ is an invariant measure with respect to $\tilde{P}^*$. Moreover, we will further show that the invariant measure is unique and ergodic and is given by
 the average $$\frac{1}{\tau_1\tau_2}\int_0^{\tau_1}\int_0^{\tau_2}\delta_{s_1}\times\delta_{s_2}\times\tilde{\rho}_{s_1,s_2}ds_1ds_2.$$
 \section{Random path and entrance measure}
 \subsection{Existence and uniqueness of random path}
 \label{Section of Existence and uniqueness of random path}
  In the stochastic differential equation (\ref{SDE}), $b: \mathbb{R}\times \mathbb{R}^d\rightarrow \mathbb{R}^d, \ \sigma: \mathbb{R}\times \mathbb{R}^d\rightarrow \mathbb{R}^{d\times d}$ are continuous functions, $W_t$ is a two-sided $\mathbb{R}^d$-valued Brownian motion on probability space $(\Omega,\mathcal{F},P)$ with $W_0=0$ 
  and
  $\xi$ is a $\mathbb{R}^d$-valued $\mathcal{F}_{-\infty}^s$-measurable random variable, where $\mathcal{F}_a^b$ is the natural filtration generated by $(W_u-W_v)_{a\leq u,v\leq b}$.
  Now we consider the following assumptions.
  
\begin{condition}
	\label{Dissipative}
	The coefficients $b, \sigma$ in SDE \eqref{SDE} satisfy the following conditions:
	\begin{description}
		\item[(1)] There exist some $\alpha>0$ such that for all $x,y\in \mathbb{R}^d$ and $t\in \mathbb{R}$,
		$$(x-y)\left(b(t,x)-b(t,y)\right)\leq -\alpha(x-y)^2;$$
		\item[(2)] There exists a constant $\beta>0$ such that for all $x,y\in \mathbb{R}^d$ and $t\in \mathbb{R}$,
		$$\|\sigma(t,x)-\sigma(t,y)\|\leq \beta |x-y|;$$
		\item[(3)] There exists $M>0$ such that
		$$\sup_{t\in \mathbb{R}}|b(t,0)|+\sup_{t\in \mathbb{R}}\|\sigma(t,0)\|\leq M;$$
	\end{description}
\end{condition}

\begin{condition}
	\label{k-th growth}
	The drift coefficient $b$ in SDE \eqref{SDE} is $\kappa$-th order growth in $x$ for some $\kappa\geq 1$, i.e. there exist $l>0$ such that for all $x\in \mathbb{R}^d$ and $t\in \mathbb{R}$,
	$$|b(t,x)|\leq l (1+|x|^{\kappa}).$$
\end{condition}
Condition \ref{k-th growth} is needed only for the purpose of perfection. For other results such as (crude) random path and the results in terms of laws including the quasi-periodic measure, the invariant measure and its ergodicity, Condition \ref{k-th growth} is not needed.

 Under Condition \ref{Dissipative}, the solution of (\ref{SDE}) exists, denoted by $X(t,s,\xi)$, and satisfies for $P-a.e.$ $\omega\in\Omega$
 $$X(t,s,\xi(\omega),\omega)=X(t,r,\omega)\circ X(r,s,\xi(\omega), \omega), \text{ for all } s\leq r\leq t.$$
 We call $u:\Delta \times \mathbb{R}^d\times \Omega \rightarrow \mathbb{R}^d$ with $u(t,s,\omega)x=X(t,s,x,\omega)$ a stochastic semi-flow, where $\Delta=\{(t,s): t\geq s, t,s\in\mathbb{R}\}$.
  \begin{definition}
  	A random path of a semi-flow $u:\Delta \times \mathbb{R}^d\times \Omega \rightarrow \mathbb{R}^d$ is a measurable map $\varphi: \mathbb{R}\times \Omega \rightarrow \mathbb{R}^d$ such that for any $t\geq s$,
  	\begin{equation}
  	\label{Random Path}
  	u(t, s, \varphi(s))=\varphi(t), \ P-a.s..
  	\end{equation}
  	We call $\varphi$ a perfect random path if equation \eqref{Random Path} holds for all $t\geq s$, $P-a.s.$ (where the exceptional set $N$ is independent of $t$ and $s$). In addition, if $u$ is generated by an SDE, we say $\varphi$ is a (perfect) random path of this SDE.
\end{definition}
  In the following, we will always use $\|\cdot\|_p$ to denote the norm in the $L^p(\Omega, dP)$ space.
  \begin{theorem}
  	\label{Existence and uniqueness of random path}
  	Assume Condition \ref{Dissipative} and $\alpha >\frac{(p-1)\beta^2}{2}$ for some $p\geq 2$. Then there exists a unique uniformly $L^p$-bounded random path $\varphi$ of SDE (\ref{SDE}), i.e. $\sup_{t\in \mathbb{R}}\|\varphi(t)\|_p<\infty$. If we further assume Condition \ref{k-th growth} and $p\geq (4+2d)\kappa$, this unique random path is perfect.
  \end{theorem}

  First we give two lemmas before we prove Theorem  \ref{Existence and uniqueness of random path}.

  \begin{lemma}
  	\label{Bounded solution}
  	Assume Condition \ref{Dissipative} and $\alpha >\frac{(p-1)\beta^2}{2}$ for some $p\geq 2$. Let $X_t^{s,\xi}$ be the solution of SDE (\ref{SDE}) with initial condition $(s, \xi)$, where $\xi \in L^p(\Omega)$. Then there exists a constant $C=C(p,\alpha, \beta, M)$ such that for all $t\geq s$,  $\|X_t^{s, \xi}\|_p^p\leq C(1+\|\xi\|_p^p)$.
  \end{lemma}
  \begin{proof}
  	We only prove this Lemma for $p>2$, since the case $p=2$ can be obtained by a similar way. For any fixed $\lambda$, applying It$\hat{\rm o}$'s formula to $e^{\lambda t}|X_t^{s, \xi}|^p$, we have
  		\begin{eqnarray*}
  		e^{\lambda t}|X_t^{s, \xi}|^p&=&e^{\lambda s}|\xi|^p+\int_{s}^{t}e^{\lambda r}|X_r^{s, \xi}|^{(p-2)}\left(\lambda |X_r^{s, \xi}|^2+pX_r^{s, \xi}\cdot b(r, X_r^{s, \xi})+\frac{p(p-1)}{2}\|\sigma(r, X_r^{s, \xi})\|^2\right)dr\\ 
  		&&+\int_{s}^{t}pe^{\lambda r}|X_r^{s, \xi}|^{(p-2)}X_r^{s, \xi}\sigma(r, X_r^{s, \xi})dW_r.
  	\end{eqnarray*}
  	In Condition \ref{Dissipative}, let $y=0$. Then for arbitrary $\epsilon >0$, by Young inequality and Condition \ref{Dissipative}
  	\begin{eqnarray*}
  		x\cdot b(t, x) &\leq -\alpha|x|^2+x\cdot b(t, 0)\\
  		&\leq -(\alpha-\epsilon)|x|^2+\frac{M^2}{4\epsilon},
  	\end{eqnarray*}
  	and
  	\begin{equation*}
    \begin{split}
    \|\sigma(t, x)\|^2 &\leq (\|\sigma(t,x)-\sigma(t,0)\|+\|\sigma(t,0)\|)^2\\
                            &\leq (\beta|x|+\|\sigma(t,0)\|)^2\\
                            &\leq (\beta^2+\epsilon)|x|^2+(\frac{\beta^2}{\epsilon}+1)M^2.
    \end{split}
    \end{equation*}  
    Since $\alpha>\frac{(p-1)\beta^2}{2}$, we can choose $\epsilon$ small enough such that $\alpha>\frac{(p-1)\beta^2}{2}+2\epsilon$ and $p(\alpha-\frac{(p-1)\beta^2}{2}-2\epsilon)>\epsilon$. Let $\lambda=p(\alpha-\frac{(p-1)\beta^2}{2}-2\epsilon)-\epsilon>0$. Then $\epsilon, \lambda$ are constants depending on $p, \alpha, \beta$. Thus there exists a constant $C(p,\alpha,\beta,M)$ depending on $p,\alpha,\beta,M$ such that
   	\begin{equation*}
      \begin{split}
      e^{\lambda t}|X_t^{s, \xi}|^p\leq 
            &e^{\lambda s}|\xi|^p+ \int_{s}^{t}e^{\lambda r}\left(-\epsilon|X_r^{s, \xi}|^p+C(p,\alpha,\beta,M)|X_r^{s, \xi}|^{(p-2)}\right)dr\\
            &+\int_{s}^{t}pe^{\lambda r}|X_r^{s, \xi}|^{(p-2)}X_r^{s, \xi}\sigma(r, X_r^{s, \xi})dW_r,
     \end{split}
	 \end{equation*}
	 where $C(p,\alpha,\beta,M)=\frac{pM^2}{4\epsilon}+\frac{p(p-1)}{2}\big(\frac{\beta^2}{\epsilon}+1\big)M^2$. Since $p-2>0$, by Young inequality
	 \begin{align*}
		C(p,\alpha,\beta,M)|X_r^{s, \xi}|^{(p-2)}&\leq \epsilon |X_r^{s, \xi}|^p+\frac{2}{p}C(p,\alpha,\beta,M)^{\frac{p}{2}}\big(\frac{\epsilon p}{p-2}\big)^{-\frac{p-2}{2}}\\
		&= \epsilon |X_r^{s, \xi}|^p+C(p,\alpha,\beta,M).
	 \end{align*}
	 Here and in the following, $C(p,\alpha,\beta,M)$ is constant, which may be different from line to line. Then we have
	 \begin{equation*}
		\begin{split}
		e^{\lambda t}|X_t^{s, \xi}|^p &\leq 
		e^{\lambda s}|\xi|^p+ C(p,\alpha,\beta,M)\int_{s}^{t}e^{\lambda r}dr+\int_{s}^{t}pe^{\lambda r}|X_r^{s, \xi}|^{(p-2)}X_r^{s, \xi}\sigma(r, X_r^{s, \xi})dW_r\\
		&\leq e^{\lambda s}|\xi|^p+ C(p,\alpha,\beta,M)e^{\lambda t}+\int_{s}^{t}pe^{\lambda r}|X_r^{s, \xi}|^{(p-2)}X_r^{s, \xi}\sigma(r, X_r^{s, \xi})dW_r.
	   \end{split}
	   \end{equation*}
    Taking expectation of both sides, we have
    \begin{align*}
		e^{\lambda t}\|X_t^{s, \xi}\|_p^p\leq e^{\lambda s}\|\xi\|_p^p+C(p,\alpha,\beta,M)e^{\lambda t}.
	\end{align*}
	Then
	\begin{align*}
		\|X_t^{s, \xi}\|_p^p\leq \|\xi\|_p^p+C(p,\alpha,\beta,M),
	\end{align*}
	which implies the desired result.
  \end{proof}

\begin{lemma}
	\label{Exponential decay}
	Assume Condition \ref{Dissipative} holds. Let $X_t^{s,\xi}$ and $X_t^{s,\eta}$ be two solutions of SDE (\ref{SDE}) with initial values $\xi$ and $\eta$ respectively, where $\xi, \eta \in L^p(\Omega)$ for some $p>1$. Then 
	$$\|X_t^{s,\xi}-X_t^{s,\eta}\|_p\leq e^{-\big(\alpha-\frac{(p-1)\beta^2}{2}\big)(t-s)}\|\xi-\eta\|_p.$$
\end{lemma}

\begin{proof}
	Note
	$$X_t^{s,\xi}-X_t^{s,\eta}=\xi-\eta+\int_{s}^{t}\left(b(r,X_r^{s,\xi})-b(r,X_r^{s,\eta})\right)dr +\int_{s}^{t}\left(\sigma(r,X_r^{s,\xi})-\sigma(r,X_r^{s,\eta})\right)dW_r.$$
	Let $\hat{X}_t:=X_t^{s,\xi}-X_t^{s,\eta}, \ \hat{b}_t=b(t,X_t^{s,\xi})-b(t,X_t^{s,\eta})$ and $\hat{\sigma}_t:=\sigma(t,X_t^{s,\xi})-\sigma(t,X_t^{s,\eta})$.
	For any fixed $\lambda$, applying It$\hat{\rm o}$'s formula to $e^{\lambda t}|\hat{X}_t|^p$, we have
	\begin{equation*}
		\begin{split}
		e^{\lambda t}|\hat{X}_t|^p
		=&e^{\lambda s}|\xi-\eta|^p
		+\int_{s}^{t}e^{\lambda r}|\hat{X}_r|^{(p-2)}\big(\lambda |\hat{X}_r|^2+p\hat{X}_r\cdot \hat{b}_r+\frac{p(p-1)}{2}\|\hat{\sigma}\|^2\big)dr\\
		&+\int_{s}^{t}pe^{\lambda r}|\hat{X}_r|^{(p-2)}\hat{X}_r\hat{\sigma}_rdW_r\\
		\leq &e^{\lambda s}|\xi-\eta|^p+\int_{s}^{t}e^{\lambda r}|\hat{X}_r|^p\bigg(\lambda-p\alpha+\frac{p(p-1)\beta^2}{2}\bigg)dr
		+\int_{s}^{t}pe^{\lambda r}|\hat{X}_r|^{(p-2)}\hat{X}_r\hat{\sigma}_rdW_r.
		\end{split}
	\end{equation*}
	Let $\lambda=p\big(\alpha-\frac{(p-1)\beta^2}{2}\big)$. Taking expectation on both sides, we have
	$$e^{\lambda t}\|X_t^{s,\xi}-X_t^{s,\eta}\|_p^p\leq e^{\lambda s}\|\xi-\eta\|_p^p.$$
	Thus the lemma follows.
\end{proof}

Now we give the proof of Theorem \ref{Existence and uniqueness of random path}

\begin{proof}[Proof of Theorem \ref{Existence and uniqueness of random path}]
	Existence: Let $s_1<s_2<t$. Then for any fixed $\xi \in L^p(\Omega)$,
	$$X_t^{s_1,\xi}=X_t^{s_2, X_{s_2}^{s_1, \xi}}.$$
	Now consider $\|X_t^{s_1, \xi}-X_t^{s_2, \xi}\|_p$. Applying Lemma \ref{Bounded solution} and Lemma \ref{Exponential decay}, we have
	\begin{equation*}
		\begin{split}
		\|X_t^{s_1, \xi}-X_t^{s_2, \xi}\|_p=&\|X_t^{s_2, X_{s_2}^{s_1, \xi}}-X_t^{s_2, \xi}\|_p\\
		\leq &e^{-\big(\alpha-\frac{(p-1)\beta^2}{2}\big)(t-s_2)}\|X_{s_2}^{s_1, \xi}-\xi\|_p\\
		\leq &e^{-\big(\alpha-\frac{(p-1)\beta^2}{2}\big)(t-s_2)}\left(\|X_{s_2}^{s_1, \xi}\|_p+\|\xi\|_p\right)\\
		\leq &C(p, \alpha, \beta, M, \|\xi\|_p)e^{-\big(\alpha-\frac{(p-1)\beta^2}{2}\big)(t-s_2)}.
		\end{split}
	\end{equation*}
	Thus there exists a $L^p$-limit of $\left(X_t^{s,\xi}\right)_{s\leq t}$ as $s\rightarrow -\infty$. By Lemma \ref{Exponential decay}, we know that this limit is independent of $\xi$. Define
	\begin{equation}
		\label{1220-0}
		\varphi(t):=L^p-\lim_{s\rightarrow -\infty}X_t^{s, \xi},
	\end{equation}
	then
	$$\|\varphi(t)\|_p\leq \limsup_{s\rightarrow -\infty}\|X_t^{s, \xi}\|_p\leq C(p, \alpha, \beta, M, \|\xi\|_p)\leq C(p,\alpha,\beta,M).$$
	Next we will prove that $\varphi$ is a random path of SDE (\ref{SDE}). For any $t\geq s\geq r$, we have
	$$u(t, s, X_s^{r,\xi})=X_t^{r,\xi}, \ P-a.s..$$
	By Lemma \ref{Exponential decay}, we know that 
	$$\|u(t, s, X_s^{r,\xi})-u(t, s, \varphi(s))\|_p\leq e^{-\big(\alpha-\frac{(p-1)\beta^2}{2}\big)(t-s)}\|X_s^{r,\xi}-\varphi(s)\|_p.$$
	It follows that for all $t\geq s$
	\begin{equation}
		\label{1220-*}
		L^p-\lim_{r\rightarrow -\infty}u(t, s, X_s^{r,\xi})=u(t, s, \varphi(s))=\varphi(t)=L^p-\lim_{r\rightarrow -\infty}X_t^{r,\xi}, \ P-a.s..
	\end{equation}
	Hence $\varphi$ is a random path of SDE \eqref{SDE}.

	Now under further Condition \ref{k-th growth} and $p\geq (4+2d)\kappa$, by $(i)$ and $(iv)$ of Lemma \ref{Continuous property of K}, we know that the solution $u(t,s,x)$ of SDE \eqref{SDE} and $\varphi(t)$ are continuous with respect to $(t,s,x)$ and $t$ $P-a.s.$, respectively. Lemma \ref{Continuous property of K} contains some key estimates needed for perfection. But in order not to interrupt the main flow of the proof of this theorem, we postpone this Lemma and its proof to the end of Section \ref{Section of Quasi-periodic path}.
	Denote 
	$$N_{s,t}:=\{\omega | u(t, s, \varphi(s, \omega),\omega)\neq\varphi(t, \omega)\}$$
	$$N_{u}:=\{\omega| u: (t,s,x)\mapsto u(t,s,x,\omega) \text{ is not continuous}\},$$
	$$N_{\varphi}:=\{\omega| \varphi: t\mapsto \varphi(t,\omega) \text{ is not continuous}\},$$
	and
	$$N=\bigcup_{t, s\in Q, t\geq s}N_{s,t}\bigcup N_{u}\bigcup N_{\varphi}$$ 
	where $Q$ is the set of all rational numbers. Since equation \eqref{1220-*} holds, we know that $P(N)=0$. Fix $\omega\in N^c$, for any $t\geq s$, we choose $\{t_n,s_n\}_{n\geq 1}$ such that $t_n\geq s_n, t_n, s_n\in Q$ and $t_n\to t, s_n\to s$, by continuity of $u(\cdot,\cdot,\cdot,\omega), \varphi(\cdot,\omega)$, we have 
	$$u(t, s, \varphi(s, \omega),\omega)=\lim_{n\rightarrow \infty}u(t_n, s_n, \varphi(s_n, \omega),\omega)=\lim_{n\rightarrow \infty}\varphi(t_n,\omega)=\varphi(t, \omega).$$
	Thus $\varphi$ is a uniformly $L^p$-bounded perfect random path of SDE (\ref{SDE}).
	
	Uniqueness: If there are two uniformly $L^p$-bounded random paths $\varphi_1, \varphi_2$ of SDE (\ref{SDE}), by Lemma \ref{Exponential decay}, we have for any $t\in \mathbb{R}$
	\begin{equation*}
	\begin{split}
	\|\varphi_1(t)-\varphi_2(t)\|_p\leq& e^{-\big(\alpha-\frac{(p-1)\beta^2}{2}\big)(t-s)}\|\varphi_1(s)-\varphi_2(s)\|_p\\
	\leq &e^{-\big(\alpha-\frac{(p-1)\beta^2}{2}\big)(t-s)}(\sup_{r\in \mathbb{R}}\|\varphi_1(r)\|_p+\sup_{r\in \mathbb{R}}\|\varphi_2(r)\|_p) \rightarrow 0 \text{ as } s\rightarrow -\infty.
	\end{split}
	\end{equation*}
	Then $\varphi_1(t)=\varphi_2(t),$ $P-a.s.$. 
	
	If there are two uniformly $L^p$-bounded perfect random paths $\varphi_1, \varphi_2$ of SDE (\ref{SDE}), denote 
	$$N_{\varphi_1}:=\{\omega| u(t,s,\varphi_1(s,\omega),\omega)\neq \varphi_1(s,\omega), \text{ for some } t\geq s\in \mathbb{R}\},$$
	and
	$$N_{\varphi_2}:=\{\omega| u(t,s,\varphi_2(s,\omega),\omega)\neq \varphi_2(s,\omega), \text{ for some } t\geq s\in \mathbb{R}\}.$$
	Since $\varphi_1, \varphi_2$ are random paths of SDE \eqref{SDE}, by Definition \ref{Random Path}, we have $P(N_{\varphi_1})=P(N_{\varphi_2})=0$. Let $N_t=\{\omega| \varphi_1(t,\omega)\neq \varphi_2(t,\omega)\}$ and 
	$$N^0=\bigcup_{n\geq 1}N_{-n}\bigcup N_{\varphi_1}\bigcup N_{\varphi_2},$$ 
	we obtain $P(N^0)=0$. Similarly fix $\omega\in (N^0)^c$, then for any $t\in \mathbb{R}$, choose $n\geq t$, we have 
	$$\varphi_1(t,\omega)=u(t,-n,\varphi_1(-n, \omega), \omega)=u(t,-n,\varphi_2(-n, \omega), \omega)=\varphi_2(t, \omega),$$
	which means $P-a.e.\ \omega\in \Omega$, 
	$$\varphi_1(t,\omega)=\varphi_2(t,w), \text{ for all } t\in \mathbb{R}.$$
\end{proof}

\subsection{Existence and uniqueness of entrance measure}
  For a semi-flow $u: \triangle\times \mathbb{R}^d\times \Omega\rightarrow \mathbb{R}^d$ with $u(t,s,x,\omega)=X_t^{s,x}(\omega)$, we define the transition $P: \triangle\times \mathbb{R}^d\times \mathcal{B}(\mathbb{R}^d)\rightarrow \mathbb{R}^+$ by $P(t,s,x,\Gamma)=P(X_t^{s,x}\in \Gamma)$ for all $t\geq s$, $x\in\mathbb{R}^d$ and $\Gamma\in\mathcal{B}(\mathbb{R}^d)$. We further define $P^*(t,s): \mathcal{P}(\mathbb{R}^d)\rightarrow\mathcal{P}(\mathbb{R}^d)$ by 
  \begin{equation}
  \label{Define measure transition P^*}
  	P^*(t,s)\mu(\Gamma)=\int_{\mathbb{R}^d}P(t,s,x,\Gamma)\mu(dx), \text{ for all } \mu\in\mathcal{P}(\mathbb{R}^d),\Gamma\in\mathcal{B}(\mathbb{R}^d).
  \end{equation}
  Here 
  $$\mathcal{P}(\mathbb{R}^d):=\{\text{all probability measures on } (\mathbb{R}^d, \mathcal{B}(\mathbb{R}^d))\}.$$

  \begin{definition}
  	We say a measure-valued map $\mu: \mathbb{R}\rightarrow \mathcal{P}(\mathbb{R}^d)$ is an entrance measure of SDE(\ref{SDE}) if $P^*(t,s)\mu_s=\mu_t$ for all $t\geq s, s\in \mathbb{R}$.
  \end{definition}
  Set
  $$\mathcal{M}^p:=\{\mu:\mathbb{R}\rightarrow \mathcal{P}(\mathbb{R}^d)| \sup_{t\in \mathbb{R}}\int_{\mathbb{R}^d}|x|^p\mu_t(dx)<\infty\}.$$
  \begin{theorem}
  	\label{Existence and uniqueness of entrance measure}
  	Assume Condition \ref{Dissipative} and $\alpha >\frac{(p-1)\beta^2}{2}$ for some $p\geq 2$. Then there exists a unique entrance measure of SDE (\ref{SDE}) in $\mathcal{M}^p$.
  \end{theorem}

  To prove Theorem \ref{Existence and uniqueness of entrance measure}, we need the following lemma.
  
  \begin{lemma}
  	\label{Two measure are the same}
  	Assume $\mu_1$ and $\mu_2$ are two probability measures on $(\mathbb{R}^d, \mathcal{B}(\mathbb{R}^d))$, and for any open set $\mathcal{O}$ we have $\mu_1(\mathcal{O})\leq \mu_2(\mathcal{O})$. Then $\mu_1=\mu_2$.
  \end{lemma}

  \begin{proof}
  	Let $\mathcal{C}:=\{\text{all open sets on }\mathbb{R}^d\}$. First we know we know that $\mu_1\leq \mu_2$ on $\mathcal{C}$.
  We now prove the opposite inequality.  For any given $\mathcal{O}\in \mathcal{C}$, $\mathcal{O}^c=\mathbb{R}^d\setminus \mathcal{O}$ is a closed set. Define 
  	$$\mathcal{O}^c_{\delta}:=\{x: dist(x, \mathcal{O}^c)<\delta\},$$
  	where $dist(x, \mathcal{O}^c)=\inf_{y\in \mathcal{O}^c}|x-y|$. Then we know that $\mathcal{O}^c_{\delta}$ is open set and $\mathcal{O}^c_{\delta}\downarrow \mathcal{O}^c$ as $\delta\downarrow 0$. Further more
  	$$\mu_1(\mathcal{O}^c)=\lim_{\delta\downarrow 0}\mu_1(\mathcal{O}^c_{\delta})\leq \lim_{\delta\downarrow 0}\mu_2(\mathcal{O}^c_{\delta})=\mu_2(\mathcal{O}^c).$$
  	Since $\mu_1$ and $\mu_2$ are probability measures, we have
  	$$1-\mu_1(\mathcal{O})\leq 1-\mu_2(\mathcal{O}),$$
  	which implies $\mu_1(\mathcal{O})\geq \mu_2(\mathcal{O})$. Hence $\mu_1\geq \mu_2$ on $\mathcal{C}$. This leads to $\mu_1=\mu_2$ on $\mathcal{C}$.
  	
  	Since $\mathcal{C}$ is a $\pi$-system and $\sigma(\mathcal{C})=\mathcal{B}(\mathbb{R}^d)$, thus $\mu_1=\mu_2$ on $\mathcal{B}(\mathbb{R}^d)$.
  \end{proof}
 
   Now we give the proof of Theorem \ref{Existence and uniqueness of entrance measure}.
   
  \begin{proof}[Proof of Theorem \ref{Existence and uniqueness of entrance measure}]
  	Existence: Applying Theorem \ref{Existence and uniqueness of random path}, we know that there exists a uniformly $L^p$-bounded random path $\varphi$ of SDE \eqref{SDE}. Let $\rho_t=\mathcal{L}(\varphi(t))$ be the law of $\varphi(t)$. Then for any $\Gamma\in\mathcal{B}(\mathbb{R}^d)$, we have
  	\begin{equation}
  	\label{Law of random path adapted}
  	\begin{split}
  	P^*(t,s)\rho_s(\Gamma)&=\int_{\mathbb{R}^d}P(t,s,x,\Gamma)\rho_s(dx)\\
  	&=\int_{\mathbb{R}^d}P(X_t^{s,x}\in\Gamma)P(\varphi(s)\in dx)\\
  	&=P(X_t^{s,\varphi(s)}\in\Gamma)\\
  	&=P(\varphi(t)\in\Gamma)\\
  	&=\rho_t(\Gamma).
  	\end{split}
  	\end{equation}
   Thus $\rho$ is an entrance measure of SDE (\ref{SDE}). And since $\varphi$ is uniformly $L^p$-bounded, then
   $$\sup_{t\in \mathbb{R}}\int_{\mathbb{R}^d}|x|^p\rho_t(dx)=\sup_{t\in \mathbb{R}}\mathbb{E}[|\varphi(t)|^p]<\infty,$$
   which means $\rho\in \mathcal{M}^p$.
  	
   Uniqueness: We aim to prove that for any entrance measure $\mu$ of SDE (\ref{SDE}) in $\mathcal{M}^p$, $\mu_t=\rho_t$ for all $t\in \mathbb{R}$.	By Lemma \ref{Two measure are the same}, we just need to prove $\rho_t(\mathcal{O})\leq \mu_t(\mathcal{O})$ for any open set $\mathcal{O}\subset \mathbb{R}^d$. Since for any $s<t$, we have
   \begin{equation*}
   \begin{split}
   \rho_t(\mathcal{O})-\mu_t(\mathcal{O})
   &=\rho_t(\mathcal{O})-\int_{\mathbb{R}^d}P(t,s,x,\mathcal{O})\mu_s(dx)\\
   &=\int_{\mathbb{R}^d}\left(\rho_t(\mathcal{O})-P(X_t^{s,x}\in\mathcal{O})\right)\mu_s(dx)\\
   &=\int_{\mathbb{R}^d}\left(P(\varphi(t)\in \mathcal{O})-P(X_t^{s,x}\in\mathcal{O})\right)\mu_s(dx).
   \end{split}
   \end{equation*}
   Define 
   $$\mathcal{O}_{\delta}:=\{x: dist(x, \mathcal{O}^c)>\delta\}.$$
   Then $\mathcal{O}_{\delta}\uparrow \mathcal{O}$ as $\delta\downarrow 0$ and
   \begin{equation*}
   \begin{split}
   P(X_t^{s,x}\in \mathcal{O})
   &=P(X_t^{s,x}-\varphi(t)+\varphi(t)\in \mathcal{O})\\
   &\geq P(\varphi(t)\in \mathcal{O}_{\delta}, |X_t^{s,x}-\varphi(t)|<\delta)\\
   &\geq P(\varphi(t)\in \mathcal{O}_{\delta})-P(|X_t^{s,x}-\varphi(t)|\geq \delta).
   \end{split}
   \end{equation*}
   Thus it turns out from the above and the Chebyshev inequality that
   \begin{equation*}
   \begin{split}
   &P(\varphi(t)\in \mathcal{O})-P(X_t^{s,x}\in \mathcal{O})\\
   &\leq P(\varphi(t)\in \mathcal{O}\setminus\mathcal{O}_{\delta})+P(|X_t^{s,x}-\varphi(t)|\geq\delta)\\
   &\leq P(\varphi(t)\in \mathcal{O}\setminus\mathcal{O}_{\delta})+\frac{1}{\delta^p}\mathbb{E}[|X_t^{s,x}-\varphi(t)|^p].
   \end{split}
   \end{equation*}
   Applying Lemma \ref{Bounded solution} and Lemma \ref{Exponential decay}, we have
   \begin{equation*}
   \begin{split}
   \mathbb{E}[|X_t^{s,x}-\varphi(t)|^p]
   &=\lim_{r\rightarrow -\infty}\mathbb{E}[|X_t^{s,x}-X_t^{r,x}|^p]\\
   &\leq \limsup_{r\rightarrow -\infty, r<s}\mathbb{E}[|X_t^{s,x}-X_t^{s,X_s^{r,x}}|^p]\\
   &\leq \limsup_{r\rightarrow -\infty, r<s}e^{-p\big(\alpha-\frac{(p-1)\beta^2}{2}\big)(t-s)}\mathbb{E}[|x-X_s^{r,x}|^p]\\
   &\leq \limsup_{r\rightarrow -\infty, r<s}C(1+|x|^p)e^{-p\big(\alpha-\frac{(p-1)\beta^2}{2}\big)(t-s)}\\
   &= C(1+|x|^p)e^{-p\big(\alpha-\frac{(p-1)\beta^2}{2}\big)(t-s)}.
   \end{split}
   \end{equation*}
   Here $C=C(p,\alpha,\beta, M)$. Then for any $\delta>0$ and $s<t$, we have
   \begin{equation*}
   \begin{split}
   \rho_t(\mathcal{O})-\mu_t(\mathcal{O})
   &=\int_{\mathbb{R}^d}\left(P(\varphi(t)\in \mathcal{O})-P(X_t^{s,x}\in\mathcal{O})\right)\mu_s(dx)\\
   &\leq \int_{\mathbb{R}^d}\left(P(\varphi(t)\in \mathcal{O}\setminus\mathcal{O}_{\delta})+\frac{1}{\delta^p}\mathbb{E}[|X_t^{s,x}-\varphi(t)|^p]\right)\mu_s(dx)\\
   &\leq P(\varphi(t)\in \mathcal{O}\setminus\mathcal{O}_{\delta})+\frac{C}{\delta^p}e^{-p\big(\alpha-\frac{(p-1)\beta^2}{2}\big)(t-s)}\int_{\mathbb{R}^d}(1+|x|^p)\mu_s(dx).
   \end{split}
   \end{equation*}
   Hence for any $\delta>0$, we have
   \begin{equation*}
   \begin{split}
   \rho_t(\mathcal{O})-\mu_t(\mathcal{O})
   &\leq P(\varphi(t)\in \mathcal{O}\setminus\mathcal{O}_{\delta})+\limsup_{s\rightarrow -\infty}\frac{C}{\delta^p}\left(1+\sup_{r\in \mathbb{R}}\int_{\mathbb{R}^d}|x|^p\mu_r(dx)\right)e^{-p\big(\alpha-\frac{(p-1)\beta^2}{2}\big)(t-s)}\\
   &\leq P(\varphi(t)\in \mathcal{O}\setminus\mathcal{O}_{\delta})=\rho_t(\mathcal{O}\setminus\mathcal{O}_{\delta}).
   \end{split}
   \end{equation*}
   Since $\mathcal{O}_{\delta}\uparrow \mathcal{O}$ as $\delta\downarrow 0$, we have
   $$\rho_t(\mathcal{O})-\mu_t(\mathcal{O})\leq \lim_{\delta\downarrow 0}\rho_t(\mathcal{O}\setminus\mathcal{O}_{\delta})=0,$$
   which implies $\rho_t(\mathcal{O})\leq \mu_t(\mathcal{O})$.
  \end{proof} 
  \begin{remark}
	\label{1222-*}
	When we consider the entrance measure, we only consider the law of random path and the perfection of $\varphi$ is not needed, thus the continuity of $\varphi$ is not needed and hence we do not need Condition \ref{k-th growth}. Then the estimates for $p=2$ in Section \ref{Section of Existence and uniqueness of random path} is adequate.
  \end{remark}
By the proof of Theorem \ref{Existence and uniqueness of random path}, we know that $\varphi(t)=L^p-\lim_{s\rightarrow -\infty}X_t^{s,x}$. Then we have the following proposition. Denote by $C_b(\mathbb{R}^d)$ the linear space of all continuous and bounded functions on $\mathbb{R}^d$.
 \begin{proposition}
 	\label{Weakly convergence of P(t,s)}
 	 The entrance measure $\rho_t$ is the limit of $P(t,s,x,\cdot)$ in $\mathcal{P}(\mathbb{R}^d)$ with weak topology, i.e. for all $f\in C_b(\mathbb{R}^d)$, we have
 	$$\lim_{s\rightarrow -\infty}\int_{\mathbb{R}^d}f(y)P(t,s,x,dy)=\int_{\mathbb{R}^d}f(y)\rho_t(dy).$$
 \end{proposition}
\begin{proof}
	Since $\int_{\mathbb{R}^d}f(y)P(t,s,x,dy)=\mathbb{E}f(X_t^{s,x})$ and $\int_{\mathbb{R}^d}f(y)\rho_t(dy)=\mathbb{E}f(\varphi(t))$, we need to prove that for all $f\in C_b(\mathbb{R}^d)$, 
	$$\lim_{s\rightarrow -\infty}\mathbb{E}f(X_t^{s,x})=\mathbb{E}f(\varphi(t)).$$
	First we prove $\limsup_{s\rightarrow -\infty}\mathbb{E}f(X_t^{s,x})\leq\mathbb{E}f(\varphi(t))$. Otherwise there exists a sequence $s_n\downarrow -\infty$ as $n \rightarrow \infty$ and a constant $\lambda=\limsup_{s\rightarrow -\infty}\mathbb{E}f(X_t^{s,x})>\mathbb{E}f(\varphi(t))$ such that $\lim_{n\rightarrow \infty}\mathbb{E}f(X_t^{s_n,x})=\lambda$. Since $\lim_{n\rightarrow \infty}\mathbb{E}[|X_t^{s_n,x}-\varphi(t)|^p]=0$, we know that there exists a subsequence $\{s_{n_k}\}\subseteq \{s_n\}$ such that $X_t^{s_{n_k},x}\xrightarrow{a.s.} \varphi(t)$ as $k \rightarrow \infty$. Thus $f(X_t^{s_{n_k},x})\xrightarrow{a.s.} f(\varphi(t))$. Then by Lebesgue's dominated convergence theorem, we have
	$$\lim_{k\rightarrow \infty}\mathbb{E}f(X_t^{s_{n_k},x})=\mathbb{E}f(\varphi(t)),$$
	which contradicts that
	$$\lim_{k\rightarrow \infty}\mathbb{E}f(X_t^{s_{n_k},x})=\lambda>\mathbb{E}f(\varphi(t)).$$
	Hence
	$$\limsup_{s\rightarrow -\infty}\mathbb{E}f(X_t^{s,x})\leq\mathbb{E}f(\varphi(t)).$$
	Similarly we can also prove that
	$$\liminf_{s\rightarrow -\infty}\mathbb{E}f(X_t^{s,x})\geq\mathbb{E}f(\varphi(t)),$$
	which completes our proof.
\end{proof}

\section{Random quasi-periodic path, quasi-periodic measure and invariant measure}
\subsection{Existence and uniqueness of random quasi-periodic path}
\label{Section of Quasi-periodic path}
   In SDE (\ref{SDE}), if we assume the coefficients $b, \sigma$ are quasi-periodic functions in time $t$, can we obtain a kind of random quasi-periodic path? What should the ``quasi-periodicity" of a random path be defined? We give the following definition.
   
   \begin{definition}
   	  \label{Quasi-periodic random path}
	A measurable path $\varphi: \mathbb{R}\times\Omega\rightarrow \mathbb{R}^d$ is called random quasi-periodic path of periods $\tau_1, \tau_2$ of a semi-flow $u$, where the reciprocals of $\tau_1$ and $\tau_2$ are rationally linearly independent, if it is a random path and there exists $\tilde{\varphi}: \mathbb{R}\times\mathbb{R}\times\Omega\rightarrow\mathbb{R}^d$ such that for any $t,s\in \mathbb{R}$,
	\begin{equation}
		\label{1226-*}
		\begin{cases}
		\tilde{\varphi}(t,t)=\varphi(t)\circ \theta_{-t},\ P-a.s.,\\
		\tilde{\varphi}(t+\tau_1,s)=\tilde{\varphi}(t,s), \ \tilde{\varphi}(t,s+\tau_2)=\tilde{\varphi}(t,s), \ P-a.s..
		\end{cases}
   	\end{equation}
   	We call $\varphi$ a perfect random quasi-periodic path if $\varphi$ is a perfect random path and equation \eqref{1226-*} holds for all $t,s\in \mathbb{R}$, $P-a.s.$. We also say $\varphi$ is a (perfect) random quasi-periodic path of an SDE if u is generated by this SDE.
   \end{definition}

   We give the quasi-periodic condition.
   
   \begin{condition}
   	\label{Quasi-periodic condition}
   	Assume that $b, \sigma$ in SDE (\ref{SDE}) are quasi-periodic functions with periods $\tau_1, \tau_2$, where the reciprocals of $\tau_1$ and $\tau_2$ are rationally linearly independent, which means there exists $\tilde{b}:\mathbb{R}\times\mathbb{R}\times\mathbb{R}^d\rightarrow\mathbb{R}^d$ and $\tilde{\sigma}:\mathbb{R}\times\mathbb{R}\times\mathbb{R}^d\rightarrow\mathbb{R}^{d\times n}$ such that $\tilde{b}(t,t,x)=b(t,x)$, $\tilde{\sigma}(t,t,x)=\sigma(t,x)$ for all $t\in\mathbb{R}, x\in\mathbb{R}^d$ satisfying
   	\begin{equation}
   	\tilde{b}(t+\tau_1,s,x)=\tilde{b}(t,s,x), \ 
   	\tilde{b}(t,s+\tau_2,x)=\tilde{b}(t,s,x),
   	\end{equation}
   	and
   	\begin{equation}
   	\tilde{\sigma}(t+\tau_1,s,x)=\tilde{\sigma}(t,s,x), \ 
   	\tilde{\sigma}(t,s+\tau_2,x)=\tilde{\sigma}(t,s,x).
   	\end{equation}
   \end{condition}
  
  \begin{condition}
  	\label{Quasi-dissipative}
  	Assume $\tilde{b}, \tilde{\sigma}$ in Condition \ref{Quasi-periodic condition} satisfy the following conditions:
  	\begin{description}
  		\item[(1)] There exist some $\alpha>0$ such that for all $x,y\in \mathbb{R}^d$ and $t,s\in \mathbb{R}$,
  		$$(x-y)\left(\tilde{b}(t,s,x)-\tilde{b}(t,s,y)\right)\leq -\alpha(x-y)^2;$$
  	    \item[(2)] There exists a constant $\beta>0$ such that for all $x,y\in \mathbb{R}^d$ and $t,s\in \mathbb{R}$,
  	    $$\|\tilde{\sigma}(t,s,x)-\tilde{\sigma}(t,s,y)\|\leq \beta |x-y|;$$
  	    \item[(3)] There exists $M>0$ such that
  	    $$\sup_{t,s\in \mathbb{R}}|\tilde{b}(t,s,0)|+\sup_{t,s\in \mathbb{R}}\|\tilde{\sigma}(t,s,0)\|\leq M;$$
  	    \item[(4)] There exist $C>0, 0<\gamma\leq 1$ such that for all $x\in \mathbb{R}^d$, $t_1,t_2,s_1,s_2\in \mathbb{R}$,
  	    $$|\tilde{b}(t_1,s_1,x)-\tilde{b}(t_2,s_2,x)|+||\tilde{\sigma}(t_1,s_1,x)-\tilde{\sigma}(t_2,s_2,x)||\leq C(|t_1-t_2|^{\gamma}+|s_1-s_2|^{\gamma}).$$
  	\end{description}
  \end{condition}

  \begin{condition}
	\label{Quasi k growth}
	We also assume $\tilde{b}$ in Condition \ref{Quasi-periodic condition} is $\kappa$-th order growth in $x$ for some $\kappa\geq 1$, i.e. there exist $l>0$ such that for all $x,y\in \mathbb{R}^d$ and $t, s\in \mathbb{R}$,
	$$|\tilde{b}(t,s,x)|\leq l (1+|x|^{\kappa}).$$
  \end{condition}
 Again Condition \ref{Quasi-dissipative} is only needed for perfection.
 Note that Conditions \ref{Quasi-periodic condition}, \ref{Quasi-dissipative} and \ref{Quasi k growth} imply Condition \ref{Dissipative} and \ref{k-th growth}. Now we give the following main theorem.
 
 \begin{theorem}
 	\label{Existence and uniqueness of quasi-periodic random path}
 	Assume Conditions \ref{Quasi-periodic condition}, \ref{Quasi-dissipative} and $\alpha>\frac{(p-1)\beta^2}{2}$ for some $p\geq 2$. Then there exists a unique uniformly $L^p$-bounded random quasi-periodic path of SDE (\ref{SDE}). If we further assume Condition \ref{Quasi k growth} and $p \geq (4+2d)\kappa\vee\frac{2\kappa}{\gamma}$, this unique random quasi-periodic path is perfect.
 \end{theorem}
  \begin{proof}
  	Uniqueness: Applying Theorem \ref{Existence and uniqueness of random path}, we know that if there exists a uniformly $L^p$-bounded (perfect) random quasi-periodic path, it must be the (perfect) random path $\varphi$ defined in Theorem \ref{Existence and uniqueness of random path}. So uniqueness holds.
  	
  	Existence: We aim to show that the random path $\varphi$ in Theorem \ref{Existence and uniqueness of random path} is indeed a random quasi-periodic path.
  	Note that the solution of SDE (\ref{SDE}) $u(t,s,x)$ can be written as for all $s\leq t, x\in \mathbb{R}^d$
  	$$u(t,s,x)=x+\int_s^t b(r,u(r,s,x))dr+\int_s^t \sigma(r,u(r,s,x))dW_r,  \ P-a.s.$$
  	Then similar to \cite{Arnold}, for a fixed real number $r$, by the measure preserving property of $\theta_{-r}$,
  	\begin{equation}
  		\label{Solution u shift time}
  		\begin{split}
  		&u(t+r,s+r,x)\circ \theta_{-r}\\
  		&=\left(x+\int_{s+r}^{t+r}b(v,u(v,s+r,x))dv+\int_{s+r}^{t+r}\sigma(v,u(v,s+r,x))dW_v\right)\circ \theta_{-r}\\
  		&=x+\left(\int_{s+r}^{t+r}b(v,u(v,s+r,x))dv\right)\circ \theta_{-r}+\left(\int_{s+r}^{t+r}\sigma(v,u(v,s+r,x))dW_v\right)\circ \theta_{-r}\\
  		&=x+\int_{s}^{t}b(v+r,u(v+r,s+r,x)\circ \theta_{-r})dv+\int_{s}^{t}\sigma(v+r,u(v+r,s+r,x)\circ \theta_{-r})dW_v\\
  		&=x+\int_{s}^{t}\tilde{b}(v+r, v+r, u(v+r,s+r,x)\circ \theta_{-r})dv\\
  		& \quad +\int_{s}^{t}\tilde{\sigma}(v+r, v+r,u(v+r,s+r,x)\circ \theta_{-r})dW_v, \ P-a.s.\\
  		\end{split}
  	\end{equation}
  	Denote $u^r(t,s,x):=u(t+r,s+r,x)\circ\theta_{-r}$, $\tilde{b}^{r_1,r_2}(t,x):=\tilde{b}(t+r_1,t+r_2,x)$ and $\tilde{\sigma}^{r_1,r_2}(t,x):=\tilde{\sigma}(t+r_1,t+r_2,x)$, then for any fixed $r$, equation (\ref{Solution u shift time}) can be written as 
  	\begin{equation}
  	\label{Solution u_r}
  	\begin{split}
  	u^r(t,s,x)
  	=x+\int_s^t\tilde{b}^{r,r}(v,u^r(v,s,x))dv+\int_s^t\tilde{\sigma}^{r,r}(v,u^r(v,s,x))dW_v,  \ P-a.s.
  	\end{split}
  	\end{equation}
	Note that the null sets in equations \eqref{Solution u shift time} and \eqref{Solution u_r} depend on $t,s,r,x$. 
  	Since Condition \ref{Quasi-dissipative} holds, then for all $r_1,r_2\in \mathbb{R}$, $\tilde{b}^{r_1,r_2}$ and $\tilde{\sigma}^{r_1,r_2}$ satisfy 
  	$$(x-y)\left(\tilde{b}^{r_1,r_2}(t,x)-\tilde{b}^{r_1,r_2}(t,y)\right)\leq -\alpha(x-y)^2,$$
  	and
  	$$\|\tilde{\sigma}^{r_1,r_2}(t,x)-\tilde{\sigma}^{r_1,r_2}(t,y)\|\leq \beta |x-y|,$$
  	for all $t\in \mathbb{R}$, $x,y\in\mathbb{R}^d$.
  	Thus the following equation
  	\begin{equation}
  	\label{Solution K_r_1,r_2}
  	\begin{split}
  	K^{r_1,r_2}(t,s,x)
  	=x+\int_s^t\tilde{b}^{r_1,r_2}(v,K^{r_1,r_2}(v,s,x))dv+\int_s^t\tilde{\sigma}^{r_1,r_2}(v,K^{r_1,r_2}(v,s,x))dW_v,
  	\end{split}
  	\end{equation}
  	has a unique solution, denoted by $K^{r_1,r_2}(t,s,x)$, where $r_1,r_2\in \mathbb{R}$ are regarded as parameters. Since $\alpha>\frac{(p-1)\beta^2}{2}$, similar to the proof of Theorem \ref{Existence and uniqueness of random path}, we know that there exist the random paths $\varphi^r(t), \varphi^{r_1,r_2}(t)$ of $u^r, K^{r_1,r_2}$ respectively such that for all $r,r_1,r_2,t\in \mathbb{R}$
  	\begin{equation}
  	\label{L2-lim K u}
  		\begin{cases}
  		\varphi^r(t)=L^p-\lim_{s\rightarrow -\infty}u^r(t,s,x)=L^p-\lim_{s\rightarrow -\infty}u(t+r,s+r,x)\circ \theta_{-r}\\
  		\varphi^{r_1,r_2}(t)=L^p-\lim_{s\rightarrow -\infty}K^{r_1,r_2}(t,s,x).
  		\end{cases}
	  \end{equation}
	Since $\varphi$ is the unique random path of SDE (\ref{SDE}), by Theorem \ref{Existence and uniqueness of random path} we have for all $r,t\in \mathbb{R}$
	\begin{equation}
		\label{2020jun5}
		\varphi^r(t)=\varphi(t+r)\circ\theta_{-r}, \ P-a.s..
	\end{equation}
  	Comparing (\ref{Solution u_r}) and (\ref{Solution K_r_1,r_2}), obviously we know that for all $t\geq s, s, r\in \mathbb{R}, x\in \mathbb{R}^d$
	\begin{eqnarray}\label{2019aug1}
	K^{r,r}(t,s,x)=u^r(t,s,x), \ P-a.s.
	\end{eqnarray} 
	and thus for all $r,t\in \mathbb{R}$,
	\begin{equation}
		\label{2020jun6}
		\varphi^{r,r}(t)=\varphi^r(t), P-a.s..
	\end{equation}
  	By quasi-periodicity of $\tilde{b}$ and $\tilde{\sigma}$, we know that $\tilde{b}^{r_1+\tau_1, r_2}=\tilde{b}^{r_1,r_2}=\tilde{b}^{r_1,r_2+\tau_2}$, $\tilde{\sigma}^{r_1+\tau_1, r_2}=\tilde{\sigma}^{r_1,r_2}=\tilde{\sigma}^{r_1,r_2+\tau_2}$. Thus it turns out that for all $t\geq s, s, r_1, r_2\in \mathbb{R}, x\in \mathbb{R}^d$
  	$$K^{r_1+\tau_1,r_2}(t,s,x)=K^{r_1,r_2}(t,s,x)=K^{r_1,r_2+\tau_2}(t,s,x), \ P-a.s..$$
	Then for all $t, r_1, r_2\in \mathbb{R}$
	\begin{equation}
		\label{2020jun7}
		\varphi^{r_1+\tau_1,r_2}(t)=\varphi^{r_1,r_2}(t)=\varphi^{r_1,r_2+\tau_2}(t), \ P-a.s..
	\end{equation}
	For all $t,s\in \mathbb{R}$, let 
	\begin{equation}
		\label{2020jun9}
		\tilde{\varphi}(t,s,\omega):=\varphi^{t,s}(0,\omega).
	\end{equation}
	It follows from \eqref{2020jun5}, \eqref{2020jun6} and \eqref{2020jun7} that for all $t\in \mathbb{R}$
	\begin{equation}
		\label{2020jun10}
		\tilde{\varphi}(t,t)=\varphi(t)\circ \theta_{-t}, \ P-a.s.,
	\end{equation}
	and for all $t,s\in \mathbb{R}$
	\begin{equation}
		\label{2020jun11}
		\tilde{\varphi}(t+\tau_1,s)=\tilde{\varphi}(t,s), \tilde{\varphi}(t,s+\tau_2)=\tilde{\varphi}(t,s), \ P-a.s..
	\end{equation}
	Therefore, this random path $\varphi$ is a random quasi-periodic path.

	Now with the Condition \ref{Quasi k growth} and assumption $p\gamma \geq 2\kappa$, by $(v-vi)$ in Lemma \ref{Continuous property of K}, $\varphi(t)\circ \theta_{-t}$ and $\tilde{\varphi}(t,s)$ are $P-a.s.$ continuous with respect to $t$ and $(t,s)$, respectively. 
	Then use the same perfection argument as in equation \eqref{1220-*}, we know that equations \eqref{2020jun10} and \eqref{2020jun11} hold for all $t\in \mathbb{R}$ and $t,s\in \mathbb{R}$ outside a null set, which ends the proof.
  \end{proof}

\begin{remark}
	\label{Solution K shift time in r}
	We can conduct similar operations as in (\ref{Solution u shift time}) and \eqref{Solution u_r}  to re-parameterised equation (\ref{Solution K_r_1,r_2}). Noticing 
	\begin{eqnarray}\label{2019aug2}
	{\tilde b}^{r_1,r_2}(v+r,\cdot)={\tilde b}^{r_1+r,r_2+r}(v,\cdot), \ \ {\tilde \sigma}^{r_1,r_2}(v+r,\cdot)={\tilde \sigma}^{r_1+r,r_2+r}(v,\cdot)
	\end{eqnarray} 
	and using the same argument as in the proof of (\ref{2019aug1}), we can conclude important property
	(\ref{2019aug3}).	
This property is similar to the shift property of the autonomous stochastic differential equations which leads to their cocycle property with a perfection argument. 
Though there is nothing similar to be said about the original SDEs due to the time dependency of the coefficients, this property holds due to  "time-invariance"
of the re-parameterised coefficients in the sense of (\ref{2019aug2}). 
\end{remark}
We state the following continuity lemma which was needed in the proof of Theorem \ref{Existence and uniqueness of random path} and Theorem \ref{Existence and uniqueness of quasi-periodic random path}. It is noted that its proof is independent of Theorem \ref{Existence and uniqueness of random path} and Theorem \ref{Existence and uniqueness of quasi-periodic random path} and their proofs. 
\begin{lemma}
	\label{Continuous property of K}
	Assume Conditions \ref{Dissipative}, \ref{Quasi-periodic condition}, \ref{Quasi-dissipative} and \ref{Quasi k growth} hold. Let $u(t,s,x), u^r(t,s,x), K^{r_1,r_2}(t,s,x)$ are the solutions of SDE \eqref{SDE}, SDE \eqref{Solution u_r} and \eqref{Solution K_r_1,r_2}, respectively. Assume further that $\alpha>\frac{(p-1)\beta^2}{2}$ for some constant $p\geq 2$, we have
	\begin{itemize}
			\item [(i)] If $p\geq (4+2d)\kappa$, $u(t,s,x)$ is continuous with respect to $(t,s,x),$ $P-a.s.$;
			\item [(ii)] If $p\gamma\wedge \frac{p}{2}\geq (3+d)\kappa$, $u^r(t,s,x)$ is continuous with respect to $(r,t,s,x),$ $P-a.s.$;
			\item [(iii)] If $p\gamma\wedge\frac{p}{2} \geq (4+d)\kappa$, $K^{r_1,r_2}(t,s,x)$ is continuous with respect to $(r_1,r_2,t,s,x)$, $P-a.s.$.
			\item [(iv)] If $p\geq 2\kappa$, $\varphi(t)$ defined in \eqref{1220-0} is continuous with respect to $t$, $P-a.s.$.
			\item [(v)] If $p\gamma\geq \kappa$, $\varphi(t)\circ \theta_{-t}$ is continuous with respect to $t$, $P-a.s.$.
			\item [(vi)] If $p\gamma\geq 2\kappa$, $\tilde{\varphi}(t,s)$ defined in \eqref{2020jun9} is continuous with respect to $(t,s)$, $P-a.s.$.
			\item [(vii)] If $p\gamma\wedge\frac{p}{2} \geq 2\kappa$, $\varphi^r(t)$ defined in \eqref{L2-lim K u} is continuous with respect to $(r,t)$, $P-a.s.$.
			\item [(viii)] If $p\gamma\wedge \frac{p}{2}\geq 3\kappa$, $\varphi^{r_1,r_2}(t)$ defined in \eqref{L2-lim K u} is continuous with respect to $(r_1,r_2,t)$, $P-a.s.$.
	\end{itemize}
	
\end{lemma}
\begin{proof}
	$(i)$ Note that 
	$$u(t,s,x)=x+\int_s^tb(v,u(v,s,x))dv+\int_s^t\sigma(v,u(v,s,x))dW_v.$$
	For any $t>t'>s$ with $|t-t'|<1$, we have
	\begin{equation}
		\label{1220-1}
		\begin{split}
		&E\left[|u(t,s,x)-u(t',s,x)|^{\frac{p}{\kappa}}\right]\\
		&=E\left[\left|\int_{t'}^tb(v,u(v,s,x))dv+\int_{t'}^t\sigma(v,u(v,s,x))dW_v\right|^{\frac{p}{\kappa}}\right]\\
		&\leq C(p,\kappa)\left(E\left[\left|\int_{t'}^tb(v,u(v,s,x))dv\right|^{\frac{p}{\kappa}}\right]+E\left[\left|\int_{t'}^t\sigma(v,u(v,s,x))dW_v\right|^{\frac{p}{\kappa}}\right]\right)\\
		&=: C(p,\kappa)(I+II).
		\end{split}
	\end{equation}
	Since $\alpha>\frac{(p-1)\beta^2}{2}$, by Lemma \ref{Bounded solution}, we know that for all $v\geq s$
	$$E[|u(v,s,x)|^{p}]\leq C(p, \alpha, \beta, M)(1+|x|^{p}).$$
	Then we can calculate
	\begin{equation}
		\label{1220-2}
		\begin{split}
		I&\leq E\left[\int_{t'}^t|b(v,u(v,s,x))|^{\frac{p}{\kappa}}dv\cdot|t-t'|^{({\frac{p}{\kappa}}-1)}\right]\\
		&\leq C(p,\kappa,l)\int_{t'}^tE[1+|u(v,s,x)|^{p}]dv\cdot|t-t'|^{({\frac{p}{\kappa}}-1)}\\
		&\leq C(p, \kappa, l,\alpha, \beta, M)(1+|x|^{p})|t-t'|^{\frac{p}{\kappa}}.
		\end{split}
	\end{equation}
	By B-D-G inequality, we have
	\begin{equation}
		\label{1220-3}
		\begin{split}
		II&\leq C(p,\kappa)E\left[\left(\int_{t'}^t\|\sigma(v,u(v,s,x))\|^2dv\right)^{\frac{p}{2\kappa}}\right]\\
		&\leq C(p,\kappa)\int_{t'}^tE[\|\sigma(v,u(v,s,x))\|^{\frac{p}{\kappa}}]dv\cdot|t-t'|^{(\frac{p}{2\kappa}-1)}\\
		&\leq C(p,\kappa, \beta, M)\int_{t'}^tE\big[1+|u(v,s,x)|^{\frac{p}{\kappa}}\big]dv\cdot|t-t'|^{(\frac{p}{2\kappa}-1)}\\
		&\leq C(p,\kappa,\alpha, \beta, M)(1+|x|^{\frac{p}{\kappa}})|t-t'|^{\frac{p}{2\kappa}}.
		\end{split}
	\end{equation}
	Then it follows from \eqref{1220-1}-\eqref{1220-3} that
	\begin{equation}
		\label{1220-4}
		E\big[|u(t,s,x)-u(t',s,x)|^{\frac{p}{\kappa}}\big]\leq C(p, \kappa, l,\alpha, \beta, M)(1+|x|^{p})|t-t'|^{\frac{p}{2\kappa}}.
	\end{equation}
	Now for any $t>s>s', x,x'\in \mathbb{R}^d$ with $|s-s'|<1$, let 
	\begin{equation*}
		\begin{cases}
			\hat{u}(t):=u(t,s,x)-u(t,s',x')\\
			\hat{b}(t):=b(t,u(t,s,x))-b(t,u(t,s',x'))\\
			\hat{\sigma}(t):=\sigma(t,u(t,s,x))-\sigma(t,u(t,s',x')).
		\end{cases}
	\end{equation*}
	Apply It$\hat{\rm o}$ formula to $|\hat{u}(t)|^{\frac{p}{\kappa}}$ on $[s,t]$, we have
	\begin{equation}
		\begin{split}
		&|\hat{u}(t)|^{\frac{p}{\kappa}}\\
		&=|\hat{u}(s)|^{\frac{p}{\kappa}}+{\frac{p}{\kappa}}\int_s^t|\hat{u}(v)|^{({\frac{p}{\kappa}}-2)}\left(\hat{u}(v)\cdot \hat{b}(v)+\frac{p/\kappa-1}{2}\|\hat{\sigma}(v)\|^2\right)dv+{\frac{p}{\kappa}}\int_s^t|\hat{u}(v)|^{({\frac{p}{\kappa}}-2)}\hat{u}(v)\hat{\sigma}(v)dW_v\\
		&\leq |\hat{u}(s)|^{\frac{p}{\kappa}}-{\frac{p}{\kappa}}\left(\alpha-\frac{(p/\kappa-1)\beta^2}{2}\right)\int_s^t|\hat{u}(v)|^{\frac{p}{\kappa}}dv+{\frac{p}{\kappa}}\int_s^t|\hat{u}(v)|^{({\frac{p}{\kappa}}-2)}\hat{u}(v)\hat{\sigma}(v)dW_v\\
		&\leq |\hat{u}(s)|^{\frac{p}{\kappa}}+{\frac{p}{\kappa}}\int_s^t|\hat{u}(v)|^{({\frac{p}{\kappa}}-2)}\hat{u}(v)\hat{\sigma}(v)dW_v.
		\end{split}
	\end{equation}
	Taking expectation on both side, we have
	\begin{equation}
		\begin{split}
		&E\big[|\hat{u}(t)|^{\frac{p}{\kappa}}\big]\leq E\big[|\hat{u}(s)|^{\frac{p}{\kappa}}\big]=E\left[\left|x-x'-\int_{s'}^sb(v,u(v,s',x'))dv-\int_{s'}^s\sigma(v,u(v,s',x'))dW_v\right|^{\frac{p}{\kappa}}\right]\\
		&\leq C(p,\kappa)\left(|x-x'|^{\frac{p}{\kappa}}+E\left[\left|\int_{s'}^sb(v,u(v,s',x'))dv\right|^{\frac{p}{\kappa}}\right]+E\left[\left|\int_{s'}^s\sigma(v,u(v,s',x'))dW_v\right|^{\frac{p}{\kappa}}\right]\right).
		\end{split}
	\end{equation}
	Similar to \eqref{1220-1}, we have
	\begin{equation}
	\label{1220-5}
	\begin{split}
		&E\big[|\hat{u}(t)|^{\frac{p}{\kappa}}\big]=E\big[|u(t,s,x)-u(t,s',x')|^{\frac{p}{\kappa}}\big]\\
		&\leq C(p, \kappa, l,\alpha, \beta, M)(1+|x'|^{p})\big(|s-s'|^{\frac{p}{2\kappa}}+|x-x'|^{\frac{p}{\kappa}}\big).
	\end{split}
	\end{equation}
	Comparing \eqref{1220-4} and \eqref{1220-5}, we conclude that
	\begin{equation}
		\label{1220-6}
		\begin{split}
			&E\big[|u(t,s,x)-u(t',s',x')|^{\frac{p}{\kappa}}\big]\\
			&\leq C(p, \kappa, l,\alpha, \beta, M)(1+|x|^{p}+|x'|^{p})\big(|t-t'|^{\frac{p}{2\kappa}}+|s-s'|^{\frac{p}{2\kappa}}+|x-x'|^{\frac{p}{\kappa}}\big).
		\end{split}
	\end{equation}
	Since $\alpha>\frac{(p-1)\beta^2}{2}$ and $p\geq (4+2d)\kappa$, we can choose $p'>p$, i.e. $\frac{p'}{2\kappa}>2+d$ such that $\alpha>\frac{(p'-1)\beta^2}{2}$ and \eqref{1220-6} holds for $p'$. Then by Kolmogorov's continuity criterion, we know that $u(t,s,x)$ is continuous with respect to $(t,s,x)$ $P-a.s.$.

	Next we prove $(iii)$, and $(ii)$ can be obtained in a similar way. Similar to the estimation of $u$ in \eqref{1220-6}, we know that for all $t,t',s,s'\in \mathbb{R}, x,x'\in \mathbb{R}^d$, $r_1,r_2\in \mathbb{R}$ with $|t-t'|<1, |s-s'|<1$,
	\begin{equation}
		\label{1220-7}
		\begin{split}
			&E\big[|K^{r_1,r_2}(t,s,x)-K^{r_1,r_2}(t',s',x')|^{\frac{p}{\kappa}}\big]\\
			&\leq C(p, \kappa, l,\alpha, \beta, M)(1+|x|^{p}+|x'|^{p})(|t-t'|^{\frac{p}{2\kappa}}+|s-s'|^{\frac{p}{2\kappa}}+|x-x'|^{\frac{p}{\kappa}}).
		\end{split}
	\end{equation}
	Now for all $r_1,r_1',r_2,r_2'\in \mathbb{R}$, let 
	\begin{equation*}
		\begin{cases}
			\hat{K}_t:=K^{r_1,r_2}(t,s,x)-K^{r'_1,r'_2}(t,s,x)\\
			\hat{b}^*_t:=\tilde{b}^{r_1,r_2}(t,K^{r_1,r_2}(t,s,x))-\tilde{b}^{r'_1,r'_2}(t,K^{r'_1,r'_2}(t,s,x))\\
			\hat{\sigma}^*_t:=\tilde{\sigma}^{r_1,r_2}(t,K^{r_1,r_2}(t,s,x))-\tilde{\sigma}^{r'_1,r'_2}(t,K^{r'_1,r'_2}(t,s,x)).
		\end{cases}
	\end{equation*}
	It follows from \eqref{Solution K_r_1,r_2} that
	\begin{align*}
		\hat{K}_t=\int_s^t\hat{b}^*_vdv+\int_s^t\hat{\sigma}^*_vdW_v.
	\end{align*}
	Now applying It$\hat{\rm o}$'s formula to $e^{\lambda t}|\hat{K}_t|^p$ for some $\lambda>0$ on $[s, t]$, we have
	\begin{equation}
		\label{2020jun1}
		\begin{split}
			e^{\lambda t}|\hat{K}_t|^p=&\int_s^t e^{\lambda v}\bigg(\lambda |\hat{K}_v|^p+p|\hat{K}_v|^{(p-2)}\hat{K}_v\cdot \hat{b}^*_v+\frac{p(p-1)}{2}|\hat{K}_v|^{(p-2)}|\hat{\sigma}^*_v|^2 \bigg) dv\\
			&+\int_s^tpe^{\lambda v}|\hat{K}_v|^{(p-2)}\hat{K}_v\hat{\sigma}^*_vdW_v.
		\end{split}
	\end{equation}
	Note that 
	\begin{equation*}
		\begin{split}
		|\hat{K}_v|^{(p-2)}\hat{K}_v\cdot \hat{b}^*_v
		=&|\hat{K}_v|^{(p-2)}\hat{K}_v\cdot\big(\tilde{b}^{r_1,r_2}(v,K^{r_1,r_2}(v,s,x))-\tilde{b}^{r'_1,r'_2}(v,K^{r_1,r_2}(v,s,x))\big)\\
		&+|\hat{K}_v|^{(p-2)}\hat{K}_v\cdot\big(\tilde{b}^{r'_1,r'_2}(v,K^{r_1,r_2}(v,s,x))-\tilde{b}^{r'_1,r'_2}(v,K^{r'_1,r'_2}(v,s,x))\big)\\
		\leq& C|\hat{K}_v|^{(p-1)}(|r_1-r'_1|^{\gamma}+|r_2-r'_2|^{\gamma})-\alpha|\hat{K}_v|^p.
		\end{split}
	\end{equation*}
	Then for arbitrary $\epsilon>0$, by Young inequality, we have
	\begin{equation}
		\label{2020jun2}
		\begin{split}
			|\hat{K}_v|^{(p-2)}\hat{K}_v\cdot \hat{b}^*_v \leq C(p,\epsilon)(|r_1-r'_1|^{p\gamma}+|r_2-r'_2|^{p\gamma})-(\alpha-\epsilon)|\hat{K}_v|^p
		\end{split}
	\end{equation}
	Similarly, we have
	\begin{equation}
		\label{2020jun3}
		\begin{split}
			|\hat{K}_v|^{(p-2)}|\hat{\sigma}^*_v|^2\leq C(p,\beta,\epsilon)(|r_1-r'_1|^{p\gamma}+|r_2-r'_2|^{p\gamma})+(\beta^2+\epsilon)|\hat{K}_v|^p
		\end{split}
	\end{equation}
	Then taking expectation on both sides of \eqref{2020jun1}, we conclude from \eqref{2020jun2} and \eqref{2020jun3} that
	\begin{equation*}
		\begin{split}
			e^{\lambda t}E[|\hat{K}_t|^p]
			\leq&\frac{1}{\lambda}C(p,\beta,\epsilon)(|r_1-r'_1|^{p\gamma}+|r_2-r'_2|^{p\gamma})e^{\lambda t}\\
			&+\bigg( \lambda-p\big(\alpha-\frac{(p-1)\beta^2}{2}-2\epsilon\big) \bigg)\int_s^te^{\lambda v}E[|\hat{K}_v|^p]dv.
		\end{split}
	\end{equation*}
	Since $\alpha>\frac{(p-1)\beta^2}{2}$, we choose $\epsilon=\frac{1}{3}\big(\alpha-\frac{(p-1)\beta^2}{2}\big)>0$ and $\lambda=p\epsilon>0$. Then we have
	\begin{equation}
		\label{2020jun4}
		E[|K^{r_1,r_2}(t,s,x)-K^{r'_1,r'_2}(t,s,x)|^p]\leq C(p,\alpha,\beta)\big(|r_1-r'_1|^{p\gamma}+|r_2-r'_2|^{p\gamma}\big).
	\end{equation}
	Hence
	\begin{equation}
		\label{1221-1}
		\begin{split}
			E\big[|K^{r_1,r_2}(t,s,x)-K^{r'_1,r'_2}(t,s,x)|^\frac{p}{\kappa}\big]&\leq\left(E\big[|K^{r_1,r_2}(t,s,x)-K^{r'_1,r'_2}(t,s,x)|^p\big]\right)^{\frac{1}{\kappa}}\\
			&\leq C(p,\kappa,\alpha,\beta)\big(|r_1-r'_1|^{\frac{p\gamma}{\kappa}}+|r_2-r'_2|^{\frac{p\gamma}{\kappa}}\big).
		\end{split}
	\end{equation}
	Then together with \eqref{1220-7}, the continuity of $K^{r_1,r_2}(t,s,x)$ can be obtained by the standard argument of using Kolmogorov's continuity criterion.

	$(iv)$ For any $t,t'\in \mathbb{R}$, it follows from \eqref{1220-0} and \eqref{1220-4} that
	\begin{equation}
		\begin{split}
		E\big[|\varphi(t)-\varphi(t')|^\frac{p}{\kappa}\big]&\leq \liminf_{s\to -\infty} E\big[|u(t,s,0)-u(t',s,0)|^\frac{p}{\kappa}\big]\\
		&\leq C(p,\kappa,l,\alpha,\beta,M)|t-t'|^{\frac{p}{2\kappa}}.
		\end{split}
	\end{equation}
	By Kolmogorov's continuity criterion, we know that $\varphi(t)$ is continuous with respect to $t$ $P-a.s.$.

	$(v-viii)$ Similarly, Comparing with \eqref{L2-lim K u}, \eqref{1220-7} and \eqref{2020jun4}, we have
	\begin{equation}
		\label{2020jun12}
		\begin{split}
			E\big[|\varphi^{r_1,r_2}(t)-\varphi^{r'_1,r'_2}(t')|^\frac{p}{\kappa}\big]&\leq \liminf_{s\to -\infty} E\big[|K^{r_1,r_2}(t,s,0)-K^{r'_1,r'_2}(t',s,0)|^\frac{p}{\kappa}\big]\\
		&\leq \liminf_{s\to -\infty} C(p,\kappa)\big\{E\big[|K^{r_1,r_2}(t,s,0)-K^{r_1,r_2}(t',s,0)|^\frac{p}{\kappa}\big]\\
		&\ \ \ \ \ \ \ \ \ \ \ \ \ \ \ \ \ \ \ \ \ 
		+E\big[|K^{r_1,r_2}(t',s,0)-K^{r'_1,r'_2}(t',s,0)|^\frac{p}{\kappa}\big]\big\}\\
		&\leq C(p,\kappa,l,\alpha,\beta,M)\big(|r_1-r'_1|^{\frac{p\gamma}{\kappa}}+|r_2-r'_2|^{\frac{p\gamma}{\kappa}}+|t-t'|^{\frac{p}{2\kappa}}\big).
		\end{split}
	\end{equation}
    Along with \eqref{2020jun5}, \eqref{2020jun6}, \eqref{2020jun9} and \eqref{2020jun12}, we conclude that
	\begin{align}
		&E\big[|\varphi(t)\theta_{-t}-\varphi(t')\theta_{-t'}|^\frac{p}{\kappa}\big]\leq C(p,\kappa,l,\alpha,\beta,M)|t-t'|^{\frac{p\gamma}{\kappa}},\\
		&E\big[|\tilde{\varphi}(t,s)-\tilde{\varphi}(t',s')|^\frac{p}{\kappa}\big]\leq C(p,\kappa,l,\alpha,\beta,M)\big(|t-t'|^{\frac{p\gamma}{\kappa}}+|s-s'|^{\frac{p\gamma}{\kappa}}\big),\label{2020jun1*}\\
		&E\big[|\varphi^r(t)-\varphi^{r'}(t')|^\frac{p}{\kappa}\big]\leq C(p,\kappa,l,\alpha,\beta,M)\big(|r-r'|^{\frac{p\gamma}{\kappa}}+|t-t'|^{\frac{p}{2\kappa}}\big).
	\end{align}
	Then we derive $(v),(vi), (vii)$ and $(viii)$ by applying Kolmogorov's continuity criterion respectively.
\end{proof}

\subsection{Existence and uniqueness of quasi-periodic measure}
  First we give the definition of the quasi-periodic probability measure as follows.
  \begin{definition}
  	\label{Definition of quasi-periodic measure}
  	We say a map $\rho: \mathbb{R}\rightarrow \mathcal{P}(\mathbb{R}^d)$ is a quasi-periodic probability measure of periods $\tau_1, \tau_2$ of SDE (\ref{SDE}), where the reciprocals of $\tau_1$ and $\tau_2$ are rationally linearly independent, if $P^*(t,s)\rho_s=\rho_t$ for all $t\geq s$, and there exists $\tilde{\rho}:\mathbb{R}\times\mathbb{R}\rightarrow \mathcal{P}(\mathbb{R}^d)$ with $\tilde{\rho}_{t,t}=\rho_t$ such that
  	\begin{equation}
  	\tilde{\rho}_{t+\tau_1,s}=\tilde{\rho}_{t,s}, \ 
  	\tilde{\rho}_{t,s+\tau_2}=\tilde{\rho}_{t,s},
  	\end{equation}
  	for all $t,s\in\mathbb{R}$.
  \end{definition}
  
\begin{theorem}
	\label{Existence of quasi-periodic measure}
	Assume Conditions \ref{Quasi-periodic condition}, \ref{Quasi-dissipative} and $\alpha>\frac{(p-1)\beta^2}{2}$ for some constant $p\geq 2$. Then there exists a unique quasi-periodic probability measure of periods $\tau_1, \tau_2$ of SDE (\ref{SDE}) in $\mathcal{M}^p$.
\end{theorem}
  
  \begin{proof}
  	Uniqueness: Applying the proof of Theorem \ref{Existence and uniqueness of entrance measure}, we know that if there exists a quasi-periodic probability measure with periods $\tau_1, \tau_2$ of SDE (\ref{SDE}) in $\mathcal{M}^p$, it must be the unique entrance measure of SDE (\ref{SDE}) defined by the law of the random path.

  	Existence: Recall equations \eqref{2020jun10} and \eqref{2020jun11} in Theorem \ref{Existence and uniqueness of quasi-periodic random path}, let 
  	\begin{equation}
  	\label{Define quasi-periodic measure by quasi-periodic path}
  		\rho_t=\mathcal{L}(\varphi(t)), \ \tilde{\rho}_{t,s}=\mathcal{L}(\tilde{\varphi}(t,s))
  	\end{equation}
	  be the laws of $\varphi(t)$ and $\tilde{\varphi}(t,s)$ respectively. Since $\varphi$ is the random path of SDE (\ref{SDE}), then by equation (\ref{Law of random path adapted}) we have $P^*(t,s)\rho_s=\rho_t$ for all $t\geq s$. Since $\theta_{-t}$ preserves probability measure $P$, then $\rho_t=\mathcal{L}(\varphi(t))=\mathcal{L}(\tilde{\varphi}(t,t))=\tilde{\rho}_{t,t}$. By the construction of $\tilde{\rho}$, we have 
	  $$\tilde{\rho}_{t+\tau_1,s}=\mathcal{L}(\tilde{\varphi}(t+\tau_1,s))=\mathcal{L}(\tilde{\varphi}(t,s))=\tilde{\rho}_{t,s}$$
  	and
  	$$\tilde{\rho}_{t,s+\tau_2}=\mathcal{L}(\tilde{\varphi}(t,s+\tau_2))=\mathcal{L}(\tilde{\varphi}(t,s))=\tilde{\rho}_{t,s}.$$
  	Also since $\varphi$ is uniformly $L^p$-bounded, then
  	$$\sup_{t\in \mathbb{R}}\int_{\mathbb{R}^d}|x|^p\rho_t(dx)=\sup_{t\in \mathbb{R}}\mathbb{E}[|\varphi(t)|^p]<\infty,$$
	which means $\rho\in \mathcal{M}^p$. Moreover, from Theorem \ref{Existence and uniqueness of random path}, Lemma \ref{Bounded solution}, Lemma \ref{Exponential decay} and the proof of Theorem \ref{Existence and uniqueness of quasi-periodic random path}, we have
	\begin{align*}
	\sup_{t,s\in \mathbb{R}}\int_{\mathbb{R}^d}|x|^p\tilde{\rho}_{t,s}(dx)=\sup_{t,s\in \mathbb{R}}\mathbb{E}[|\varphi^{t,s}(0)|^p]<\infty,
	\end{align*}
	which shows that $\tilde{\rho}\in \mathcal{M}^p$.
  \end{proof}

  \begin{remark}
	  Similar to Remark \ref{1222-*}, the continuity of $\tilde{\varphi}$ and Condition \ref{Quasi k growth} are not needed when we consider the quasi-periodic measure. From Theorem \ref{Existence and uniqueness of quasi-periodic random path} and its proof, we know that the estimates for $p=2$ in Section 3.1 is also adequate.
  \end{remark}
  
  \begin{example}[Ornstein-Uhlenbeck equation]
  We include the following example with a number of reasons. First, O-U process is one of the simplest stochastic process
  that one would analyse for new concepts.  Second, it is instructive and does illustrate clearly 
  the idea of random quasi-periodicity and two kinds of 
  formulations as well as their relation. Third, the formulae for its random quasi-periodic path and quasi-periodic measure can be written down explicitly. 
  Last, but not least, this equation is relevant in various different applications e.g. modelling energy consumptions or temperature variants with two obvious daily and seasonal periodicities.   
  	 
	 The Ornstein-Uhlenbeck process with mean reversion of single-period was used in modelling electricity prices (\cite{BKM07},\cite{LS02}), daily temperature 
(\cite{BB07}), biological neurouns (\cite{IDL14}) etc. The quasi-periodic O-U process we introduce here allows a feature of multiple periods which is natural
in many real world situations e.g energy consumptions, temperature, business cycles, economics cycles. While it is not the purpose 
of this paper to study these interesting applied problems in their specific contexts, our work in this paper provides a mathematical
theory of random quasi-periodicity for this purpose.
	 
	 Here we consider the following mean reversion multidimensional Ornstein-Uhlenbeck equation on $\mathbb{R}^d$
  	\begin{equation}
  	\label{O-U equation}
  	dX_t=(S(t)-AX_t)dt+\sigma(t)dW_t
  	\end{equation}
  	where $S(t), \sigma(t)$ are deterministic quasi-periodic functions with periods $\tau_1,\tau_2$ and $A\in S_d$ with $A>0$, which means that $A$ is a symmetrical matrix with positive eigenvalues $\{\lambda_n\}_{n=1}^{d}$. 
	 The analysis is given as follows.
	
	Applying It$\hat{o}$'s formula to $e^{tA}X_t$, we have
  	\begin{equation}
  	X_t=e^{-(t-s)A}X_s+\int_{s}^{t}e^{-(t-r)A}S(r)dr+\int_{s}^{t}e^{-(t-r)A}\sigma(r)dW_r \quad t\geq s.
  	\end{equation}
  	Let 
  	\begin{eqnarray}
	\label{2019aug4}
	\varphi(t):=\int_{-\infty}^{t}e^{-(t-r)A}S(r)dr+\int_{-\infty}^{t}e^{-(t-r)A}\sigma(r)dW_r.
	\end{eqnarray}
  	Then we have 
  	\begin{equation}
  	\varphi(t)=e^{-(t-s)A}\varphi(s)+\int_{s}^{t}e^{-(t-r)A}S(r)dr+\int_{s}^{t}e^{-(t-r)A}\sigma(r)dW_r,
  	\end{equation}
  	which means that $\varphi$ is a random path of SDE (\ref{O-U equation}). Next we will show that $\varphi$ is also a random quasi-periodic path. We first rewrite $\varphi(t)$ by 
  	\begin{equation}
  	\begin{split}
  	\varphi(t, \omega)
  	&=\int_{-\infty}^{t}e^{-(t-r)A}S(r)dr+\left[\int_{-\infty}^{t}e^{-(t-r)A}\sigma(r)dW_r\right](\omega)\\
  	&=\int_{-\infty}^{0}e^{rA}S(r+t)dr+\left[\int_{-\infty}^{0}e^{rA}\sigma(r+t)dW_r\right](\theta_t\omega).
  	\end{split}
  	\end{equation}
  	Since $S,\sigma$ are quasi-periodic functions with periods $\tau_1,\tau_2$, then there exist $\tilde{S},\tilde{\sigma}$ such that $S(t)=\tilde{S}(t,t)$ and $\sigma(t)=\tilde{\sigma}(t,t)$ and 
  	\begin{equation}
  	\begin{cases}
  	\tilde{S}(t+\tau_1,s)=\tilde{S}(t,s)=\tilde{S}(t,s+\tau_2)\\
  	\tilde{\sigma}(t+\tau_1,s)=\tilde{\sigma}(t,s)=\tilde{\sigma}(t,s+\tau_2).
  	\end{cases}
  	\end{equation}
  	Then we have
  	\begin{equation*}
  	\begin{split}
  	\varphi(t, \omega)
  	&=\int_{-\infty}^{0}e^{rA}\tilde{S}(r+t,r+t)dr+\left[\int_{-\infty}^{0}e^{rA}\tilde{\sigma}(r+t,r+t)dW_r\right](\theta_t\omega).
  	\end{split}
  	\end{equation*}
  	Let
  	\begin{eqnarray}
	\label{2019aug5}
	\tilde{\varphi}(t,s,\omega)=\int_{-\infty}^{0}e^{rA}\tilde{S}(r+t,r+s)dr+\left[\int_{-\infty}^{0}e^{rA}\tilde{\sigma}(r+t,r+s)dW_r\right](\omega).
	\end{eqnarray}
  	Then we have $\varphi(t,\theta_{-t}\omega)=\tilde{\varphi}(t,t,\omega)$ and
  	\begin{equation}
  	\tilde{\varphi}(t+\tau_1,s,\omega)=\tilde{\varphi}(t,s,\omega), \ 
  	\tilde{\varphi}(t,s+\tau_2,\omega)=\tilde{\varphi}(t,s,\omega),
  	\end{equation}
  	which shows that $\varphi$ is a random quasi-periodic path of SDE (\ref{O-U equation}) with periods $\tau_1,\tau_2$.
  	
  	Let $\rho_t=\mathcal{L}(\varphi(t))$. By Theorem \ref{Existence of quasi-periodic measure}, we know that $\rho_t$ is the unique quasi-periodic probability measure with periods $\tau_1,\tau_2$ of SDE (\ref{O-U equation}). Moreover, from (\ref{2019aug4}), 
  	we know that
  	$$\rho_t(\cdot)=\mathcal{N}\left(\int_{-\infty}^{t}e^{-(t-r)A}S(r)dr, \int_{-\infty}^{t}e^{-(t-r)A}\sigma(r)\sigma(r)^Te^{-(t-r)A}dr\right)(\cdot),$$
  	where $\mathcal{N}$ is the multivariate normal distribution. 
	Let $\tilde \rho_{t,s}=\mathcal{L}(\tilde \varphi(t,s))$. Then from (\ref{2019aug5}), 
  	we know that
  	$$\tilde \rho_{t,s}(\cdot)=\mathcal{N}\left(\int_{-\infty}^{0}e^{rA}\tilde S(r+t,r+s)dr, \int_{-\infty}^{0}e^{rA}(\tilde \sigma\tilde 
	\sigma^T)(r+t,r+s)e^{rA}dr\right)(\cdot).$$
It is obvious that $\rho_t=\tilde \rho_{t,t}$. 

In Subsection \ref{2019aug6}, we will develop a way to lift a quasi-periodic stochastic flow to the cylinder $[0,\tau_1)\times [0,\tau_2)\times {\mathbb R}^d$ and prove  $\tilde \mu_{t,s}=\delta_t\times\delta_s\times \tilde \rho_{t,s}$ is a
quasi-periodic measure. This setup will enable us to prove that the average ${1\over {\tau_1\tau_2}}\int _0^{\tau_1}\int_0^{\tau_2}\tilde \mu_{t,s}dtds$ is an ergodic invariant measure on the cylinder. Our result also implies that for this particular case, it is the unique ergodic invariant measure for the lifted quasi-periodic Ornstein-Uhlenbeck process.
  \end{example}

\subsection{The lift and invariant measure}\label{2019aug6}

In Section \ref{Section of Quasi-periodic path}, we have the existence and uniqueness of random quasi-periodic path, and in this case, we will lift the semi-flow $u$ and obtain an invariant measure.
Consider the cylinder $\tilde{\mathbb X}=[0, \tau_1) \times [0, \tau_2)\times\mathbb{R}^d$ with the following metric
$$d(\tilde{x}, \tilde{y})=d_1(t_1,s_1)+d_2(t_2,s_2)+|x-y|,  \text{ for all } \tilde{x}=(t_1,t_2,x), \tilde{y}=(s_1,s_2,y) \in \tilde{\mathbb X},$$
where $d_1,d_2$ are the metrics on $[0,\tau_1), [0,\tau_2)$ defined by
$$d_i(t_i,s_i)=\min(|t_i-s_i|, \tau_i-|t_i-s_i|), \text{ for all } t_i,s_i\in [0,\tau_i), i=1,2.$$
Denote by $\mathcal{B}(\tilde{\mathbb X})$ the Borel measurable set on $\tilde{\mathbb X}$ deduced by metric $d$. Then we have the following lemma. Note the perfection of $\ K^{s_1,s_2}(t,0,x,\cdot)$ is needed but that of $\tilde{\varphi}$ is not needed.

\begin{lemma}
	\label{Lemma of cocycle of the lift semiflow}
	Assume Conditions \ref{Quasi-periodic condition}, \ref{Quasi-dissipative} hold. We lift the semi-flow $u: \triangle\times \mathbb{R}^d\times \Omega\rightarrow \mathbb{R}^d$ to a random dynamical system on a cylinder $\tilde{\mathbb X}=[0, \tau_1) \times [0, \tau_2)\times\mathbb{R}^d$ by the following:
	$$\tilde{\Phi}(t,\omega)(s_1,s_2,x)=(t+s_1\mod \tau_1,\ t+s_2 \mod\tau_2,\ K^{s_1,s_2}(t,0,x,\omega)),$$
	where $K^{r_1,r_2}$ is the solution of (\ref{Solution K_r_1,r_2}). Then $\tilde{\Phi}: \mathbb{R}^+\times \tilde{\mathbb X}\times\Omega\rightarrow \tilde{\mathbb X}$ is a cocycle on $\tilde{\mathbb X}$ over the metric dynamical system $(\Omega, \mathcal{F}, P, (\theta_t)_{t\in \mathbb{R}})$. If we further assume that Condition \ref{Quasi k growth} holds and $\alpha>\frac{(p-1)\beta^2}{2}$ for some $p \geq (10+2d)\kappa\vee\frac{(5+d)\kappa}{\gamma}$, then $\tilde{\Phi}$ is a perfect cocycle on $\tilde{\mathbb X}$.
	
	Moreover, assume $\varphi: \mathbb{R}\times \Omega\rightarrow \mathbb{R}^d$ is a random path of the semi-flow u. Then $\tilde{Y}: \mathbb{R} \times \Omega\rightarrow \tilde{\mathbb X}$ defined by 
	$$\tilde{Y}(s,\omega)=(s\mod \tau_1,\ s\mod\tau_2,\ \varphi(s,\omega))$$
	is a random path of the cocycle $\tilde{\Phi}$ on $\tilde{\mathbb X}$.
\end{lemma}
\begin{proof}
	We first prove that $\tilde{\Phi}$ is a cocycle on $\tilde{\mathbb X}$. Note $K^{r_1,r_2}$ is periodic in $r_1,r_2$ with periods $\tau_1,\tau_2$. It follows that for any $(s_1,s_2,x)\in \tilde{\mathbb X}, t,s\in \mathbb{R}^+$, we have
	\begin{eqnarray*}
	&&\tilde{\Phi}(t, \theta_s\omega)\circ\tilde{\Phi}(s,\omega)(s_1,s_2,x)\\
	&=&\tilde{\Phi}(t,\theta_s\omega)(s+s_1\mod\tau_1,\ s+s_2\mod\tau_2,\ K^{s_1,s_2}(s,0,x,\omega))\\
	&=&(t+s+s_1\mod\tau_1,\ t+s+s_2\mod\tau_2,\ K^{s+s_1, s+s_2}(t,0,K^{s_1,s_2}(s,0,x,\omega), \theta_s\omega)).
	\end{eqnarray*}
	Now we compute the $K^{s+s_1, s+s_2}(t,0,K^{s_1,s_2}(s,0,x,\omega), \theta_s\omega)$ term. By Remark \ref{Solution K shift time in r}, we know that equation \eqref{2019aug3} holds $P-a.s.$, i.e. for all $r_1,r_2,r\in \mathbb{R}$, $t\geq s$,
	\begin{equation*}
	K^{r_1,r_2}(t+r,s+r,x,\omega)=K^{r+r_1,r+r_2}(t,s,x,\theta_r\omega), \ P-a.e. \text{ on }\omega.
	\end{equation*}
	Then we have for all $s_1,s_2\in \mathbb{R}, t,s\in \mathbb{R}^+$
	\begin{align*}
		K^{s+s_1, s+s_2}(t,0,K^{s_1,s_2}(s,0,x,\omega), \theta_s\omega)&=K^{s_1,s_2}(t+s,s,K^{s_1,s_2}(s,0,\omega),\omega)\\
		&=K^{s_1,s_2}(t+s,0,\omega), \ P-a.e. \text{ on } \omega,
	\end{align*}
	i.e. the set $N^{s_1,s_2}_{s,t}:=\{\omega| K^{s+s_1, s+s_2}(t,0,K^{s_1,s_2}(s,0,x,\omega), \theta_s\omega)\neq K^{s_1,s_2}(t+s,0,\omega)\}$ is a null set. Then for any fixed $t\geq s, (s_1,s_2,x)\in \tilde{\mathbb{X}}$, 
	\begin{equation*}
		\begin{split}
		&\tilde{\Phi}(t, \theta_s\omega)\circ\tilde{\Phi}(s,\omega)(s_1,s_2,x)\\
		&=(t+s+s_1\mod\tau_1,t+s+s_2\mod\tau_2, K^{s+s_1, s+s_2}(t,0,K^{s_1,s_2}(s,0,x,\omega), \theta_s\omega))\\
		&=(t+s+s_1\mod\tau_1,t+s+s_2\mod\tau_2, K^{s_1,s_2}(t+s, 0, x, \omega))\\
		&=\tilde{\Phi}(t+s,\omega)(s_1,s_2,x), P-a.e. \ \omega,
		\end{split}
	\end{equation*}
	which means that $\tilde{\Phi}$ is a cocycle on $\tilde{\mathbb{X}}$.

	If Condition \ref{Quasi k growth} holds and $\alpha>\frac{(p-1)\beta^2}{2}$ for some $p \geq (10+2d)\kappa\vee\frac{(5+d)\kappa}{\gamma}$, 
	from \eqref{2019aug3}, \eqref{1220-7} and \eqref{1221-1} in Lemma \ref{Continuous property of K}, we know that there exists $C=C(p, \kappa, l,\alpha, \beta, M)(1+|x|^p+|x'|^p)$ such that
	\begin{equation*}
		\begin{split}
			&E\left[\left|K^{r+r_1,r+r_2}(t,s,x)\circ\theta_r-K^{r'+r_1',r'+r_2'}(t',s',x')\circ\theta_{r'}\right|^{\frac{p}{\kappa}}\right]\\
			&\leq C\big(|r_1-r'_1|^{\frac{p\gamma}{\kappa}}+|r_2-r'_2|^{\frac{p\gamma}{\kappa}}+|r-r'|^{\frac{p}{2\kappa}}+|t-t'|^{\frac{p}{2\kappa}}+|s-s'|^{\frac{p}{2\kappa}}+|x-x'|^{\frac{p}{\kappa}}\big).
		\end{split}
	\end{equation*}
	Then by Kolmogorov's continuity criterion, $K^{r+r_1,r+r_2}(t,s,x,\theta_r\cdot)$ is continuous with respect to $(r,r_1,r_2,t,s,x)$ $P-a.s.$. By $(iii)$ in Lemma \ref{Continuous property of K}, we also know that $K^{r_1,r_2}(t,s,x)$ is continuous with respect to $(r_1,r_2,t,s,x)$ $P-a.s.$.
	Denote 
	$$N^1:=\{\omega| K: (r_1,r_2,t,s,x)\mapsto K^{r_1,r_2}(t,s,x,\omega) \text{ is not continuous}\}$$
	$$N^2:=\{\omega| K: (r,r_1,r_2,t,s,x)\mapsto K^{r+r_1,r+r_2}(t,s,x,\theta_{r}\omega) \text{ is not continuous}\}$$
	and
	$$N_K:=\bigcup_{s_1,s_2\in Q, t, s\in Q^+}N^{s_1,s_2}_{s,t}\bigcup N^1 \bigcup N^2.$$
	Then $P(N^1)=P(N^2)=0$ and hence $P(N_K)=0$. Fix $\omega\in N_K^c$, for any $s_1,s_2\in \mathbb{R}, t,s\in \mathbb{R}^+$, we choose $\{s_1^n,s_2^n,t^n,s^n\}_{n\geq 1}$ such that $s_1^n,s_2^n\in Q, t^n, s^n\in Q^+$ and $s_1^n\to s_1, s_2^n\to s_2, t^n\to t, s^n\to s$, then we have
	\begin{equation*}
	\begin{split}
	&\tilde{\Phi}(t, \theta_s\omega)\circ\tilde{\Phi}(s,\omega)(s_1,s_2,x)\\
	=&(t+s+s_1\mod\tau_1,t+s+s_2\mod\tau_2, K^{s+s_1, s+s_2}(t,0,K^{s_1,s_2}(s,0,x,\omega), \theta_s\omega))\\
	=&(t+s+s_1\mod\tau_1,t+s+s_2\mod\tau_2, \lim_{n\rightarrow \infty}K^{s^n+s^n_1, s^n+s^n_2}(t^n,0,K^{s^n_1,s^n_2}(s^n,0,x,\omega), \theta_{s^n}\omega))\\
	=&(t+s+s_1\mod\tau_1,t+s+s_2\mod\tau_2, \lim_{n\rightarrow \infty}K^{s^n_1, s^n_2}(t^n+s^n,0,x,\omega))\\
	=&(t+s+s_1\mod\tau_1,t+s+s_2\mod\tau_2, K^{s_1,s_2}(t+s, 0, x, \omega))\\
	=&\tilde{\Phi}(t+s,\omega)(s_1,s_2,x),
	\end{split}
	\end{equation*}
	which implies the perfect cocycle property of $\tilde{\Phi}$.
	
	Next, assume $\varphi$ is a random path of the semi-flow $u$ and $\tilde{Y}(s,\omega)=(s\mod \tau_1, s\mod\tau_2, \varphi(s,\omega))$. Denote $N_{\varphi}:=\{\omega| u(t,s,\varphi(s,\omega),\omega)\neq \varphi(t,\omega) \text{ for all } t\geq s\}$, then $P(N_{\varphi})=0$. Fix $\omega\in N_{\varphi}^c\cap N_K^c$,
	\begin{equation*}
	\begin{split}
	\tilde{\Phi}(t,\theta_s\omega)\tilde{Y}(s,\omega)
	    =&(t+s\mod \tau_1, t+s\mod\tau_2, K^{s, s}(t, 0, \varphi(s, \omega), \theta_s\omega))\\
	    =&(t+s\mod \tau_1, t+s\mod\tau_2, u^s(t, 0, \varphi(s, \omega), \theta_s\omega))\\
	    =&(t+s\mod \tau_1, t+s\mod\tau_2, u(t+s, s, \varphi(s, \omega), \theta_{-s}\theta_s\omega))\\
	    =&(t+s\mod \tau_1, t+s\mod\tau_2, \varphi(t+s,\omega))\\
	    =&\tilde{Y}(t+s,\omega),
	\end{split}
	\end{equation*}
	which means $\tilde{Y}$ is a random path of the cocycle $\tilde{\Phi}$ on $\tilde{\mathbb X}$.
\end{proof}

Consider the Markovian transition $\tilde{P}:\mathbb{R}^+\times \tilde{\mathbb X}\times \mathcal{B}(\tilde{\mathbb X})\rightarrow [0,1]$ generated by the cocycle $\tilde{\Phi}$, i.e.,
$$\tilde{P}(t, (s_1,s_2,x), \tilde{\Gamma})=P(\omega: \tilde{\Phi}(t,\omega)(s_1,s_2,x)\in \tilde{\Gamma}),$$
for all $t\in \mathbb{R}^+, (s_1,s_2,x)\in \tilde{\mathbb X}, \tilde{\Gamma}\in \mathcal{B}(\tilde{\mathbb X})$. Similarly, for any $\tilde{\mu}\in \mathcal{P}(\tilde{\mathbb X})$, we define
$$\tilde{P}^*_t\tilde{\mu}(\tilde{\Gamma})=\int_{\tilde{\mathbb X}}\tilde{P}(t,(s_1,s_2,x),\tilde{\Gamma})\tilde{\mu}(ds_1\times ds_2\times dx).$$
Then we have the following theorem.
\begin{theorem}
	If $\rho: \mathbb{R}\rightarrow \mathcal{P}(\mathbb{R}^d)$ is the entrance measure of semi-group $P^*$, i.e. $P^*(t,s)\rho_s=\rho_t$, then $\tilde{\mu}: \mathbb{R}\rightarrow \mathcal{P}(\tilde{\mathbb X})$ defined by 
	$$\tilde{\mu}_t=\delta_{t\mod \tau_1}\times \delta_{t\mod \tau_2}\times \rho_t$$
	is an entrance measure of semi-group $\tilde{P}^*$, i.e., 
	$$\tilde{P}^*_t\tilde{\mu}_s=\tilde{\mu}_{t+s}.$$
	Moreover, $\tilde{\mu}$ is also a quasi-periodic measure.
\end{theorem}
\begin{proof}
	For any $\tilde{\Gamma}\in \mathcal{B}(\tilde{\mathbb X})$, let $\tilde{\Gamma}_s:=\{x\in \mathbb{R}^d| (s\mod \tau_1, s\mod \tau_2,x)\in \tilde{\Gamma}\}$. Then we have
	\begin{equation*}
	\begin{split}
	\tilde{P}^*_t\tilde{\mu}_s(\tilde{\Gamma})
	&=\int_{\tilde{\mathbb X}}\tilde{P}(t, (s_1,s_2,x), \tilde{\Gamma})\tilde{\mu}_s(ds_1\times ds_2\times dx)\\
	&=\int_{\mathbb{R}^d}\tilde{P}(t, (s\mod \tau_1,\ s\mod \tau_2,x), \tilde{\Gamma})\rho_s(dx)\\
	&=\int_{\mathbb{R}^d}P(\omega: \tilde{\Phi}(t,\omega)(s\mod \tau_1,\ s\mod \tau_2,x)\in \tilde{\Gamma})\rho_s(dx)\\
	&=\int_{\mathbb{R}^d}P(\omega: (t+s\mod \tau_1,\ t+s\mod \tau_2, u^s(t,0,x,\omega))\in \tilde{\Gamma})\rho_s(dx)\\
	&=\int_{\mathbb{R}^d}P(\omega: u(t+s,s,x,\theta_{-s}\omega)\in \tilde{\Gamma}_{t+s})\rho_s(dx)\\
	&=\int_{\mathbb{R}^d}P(\omega: u(t+s,s,x,\omega)\in \tilde{\Gamma}_{t+s})\rho_s(dx)\\
	&=\int_{\mathbb{R}^d}P(t+s,s,x, \tilde{\Gamma}_{t+s})\rho_s(dx)\\
	&=P^*(t+s,s)\rho_s(\tilde{\Gamma}_{t+s})\\
	&=\rho_{t+s}(\tilde{\Gamma}_{t+s})=\tilde{\mu}_{t+s}(\tilde{\Gamma}).
	\end{split}
	\end{equation*}
	Moreover, let
	\begin{equation}
	\label{Quasi-periodic measure after lift}
		\hat{\mu}_{s_1,s_2}=\delta_{s_1\mod\tau_1}\times\delta_{s_2\mod\tau_2}\times\tilde{\rho}_{s_1,s_2},
	\end{equation}
	we know that $\tilde{\mu}_s=\hat{\mu}_{s_1,s_2}$ and
	\begin{equation}
	\hat{\mu}_{s_1+\tau_1,s_2}=\hat{\mu}_{s_1,s_2}, \ 
	\hat{\mu}_{s_1,s_2+\tau_2}=\hat{\mu}_{s_1,s_2},
	\end{equation}
	which completes our proof.
\end{proof}

For the above entrance measure $\tilde{\mu}$, set
$$\bar{\tilde{\mu}}_T:=\frac{1}{T}\int_{0}^{T}\tilde{\mu}_sds$$
and
\begin{equation}
 \label{Tight measure set}
	\mathcal{M}:=\{\bar{\tilde{\mu}}_T: T\in \mathbb{R}^+\}.
\end{equation}

We have the following lemma.
\begin{lemma}
	\label{Lemma of tight measure set}
	Assume Conditions \ref{Quasi-periodic condition}, \ref{Quasi-dissipative} and $\alpha>\frac{\beta^2}{2}$. Then $\mathcal{M}$ is tight, and hence is weakly compact in $\mathcal{P}(\tilde{\mathbb X})$.
\end{lemma}
\begin{proof}
	We just need to prove that for any $\epsilon>0$, there exists a compact set $\tilde{\Gamma}_{\epsilon}\in \mathcal{B}(\tilde{\mathbb X})$ such that for all $T\in \mathbb{R}^+$, we have
	$$\bar{\tilde{\mu}}_T({\tilde{\Gamma}_{\epsilon}})>1-\epsilon.$$
	Since the entrance measure $\rho_t$ is the law of the $L^2$-bounded random path $\varphi(t)$, then $\{\rho_t: t\in \mathbb{R}\}$ is tight because 
	\begin{equation}
	\begin{split}
	\rho_t(\bar{B}_{N}(0))&=P(|\varphi(t)| \leq N)\\
	                       &=1-P(|\varphi(t)| > N)\\
	                       &\geq 1-\frac{\|\varphi(t)\|_2^2}{N^2}\\
	                       &\geq 1-\frac{\sup_{t\in \mathbb{R}}\|\varphi(t)\|_2^2}{N^2}.
	\end{split}
	\end{equation}
     Then for the given $\epsilon>0$, there exists a compact set $\Gamma_{\epsilon}\subset \mathbb{R}^d$ such that for all $t\in \mathbb{R}$,
	$$\rho_t(\Gamma_{\epsilon})>1-\epsilon.$$
	It is well-known that $[0,\tau_1), [0,\tau_2)$ are both homeomorphic to the circle $S^1$ under metrics $d_1,d_2$ respectively. Hence they are compact and $\tilde{\Gamma}_{\epsilon}=[0,\tau_1)\times [0,\tau_2) \times \Gamma_{\epsilon}$ is compact on $\tilde{\mathbb X}$. Moreover
	\begin{equation}
		\begin{split}
		\bar{\tilde{\mu}}_T(\tilde{\Gamma}_{\epsilon})=\frac{1}{T}\int_0^T\tilde{\mu}_s(\tilde{\Gamma}_{\epsilon})ds=\frac{1}{T}\int_0^T\rho_s(\Gamma_{\epsilon})ds>1-\epsilon,
		\end{split}
	\end{equation}
	which completes our proof.
\end{proof}

For any $f\in C^0(\tilde{\mathbb X})$, which is defined as the collection of $\mathcal{B}(\tilde{\mathbb X})$ measurable functions, we define
\begin{equation}
\label{P^*_t act on functions}
\tilde{P}_tf(\tilde{x})=\int_{\tilde{\mathbb X}}\tilde{P}(t,\tilde{x},d\tilde{y})f(\tilde{y}), \text{ for any } \tilde{x}\in \tilde{\mathbb X}.
\end{equation}
We have the following Feller property of the semi-group $\tilde{P}_t, t\geq 0$.

\begin{proposition}
	\label{Feller property of P^*}
	Assume Conditions \ref{Quasi-periodic condition}, \ref{Quasi-dissipative} and $\alpha>\frac{\beta^2}{2}$. Then the semi-group $\tilde{P}_t, t\geq 0$, defined by (\ref{P^*_t act on functions}), is Feller, i.e. for all $f\in C_b(\tilde{\mathbb X})$, $\tilde{P}_tf\in C_b(\tilde{\mathbb X})$.
\end{proposition}

\begin{proof}
	Obviously $\|\tilde{P}_tf\|_{\infty}\leq \|f\|_{\infty}$, then we just need to prove that $\tilde{P}_tf$ is continuous. It is sufficient to prove that for any sequence $\tilde{x}_n=(r_1^n,r_2^n,x_n), \tilde{x}=(r_1,r_2,x)\in \tilde{\mathbb X}$ with $\tilde{x}_n\xrightarrow{n\rightarrow \infty} \tilde{x}$, we have $\tilde{P}_tf(\tilde{x}_n)\xrightarrow{n\rightarrow \infty} \tilde{P}_tf(\tilde{x})$. Since
	\begin{equation}
	\begin{split}
	\tilde{P}_tf(\tilde{x})
	&=\int_{[0, \tau_1) \times [0, \tau_2)\times\mathbb{R}^d}\tilde{P}(t,(r_1,r_2,x),ds_1\times ds_2\times dy)f(s_1,s_2,y)\\
	&=\int_{[0, \tau_1) \times [0, \tau_2)\times\mathbb{R}^d}P(\tilde{\Phi}(t,\cdot)(r_1,r_2,x) \in ds_1\times ds_2\times dy)f(s_1,s_2,y)\\
	&=\int_{\mathbb{R}^d}P(K^{r_1,r_2}(t,0,x) \in dy)f(t+r_1\mod \tau_1,t+r_2\mod \tau_2,y)\\
	&=\mathbb{E}f(t+r_1\mod \tau_1,t+r_2\mod \tau_2,K^{r_1,r_2}(t,0,x)).
	\end{split}
	\end{equation}
	Let $f_t(r_1,r_2,x):=f(t+r_1\mod \tau_1,t+r_2\mod \tau_2, x)$. Then we have
	\begin{equation}
	\begin{split}
	|\tilde{P}_tf(\tilde{x}_n)-\tilde{P}_tf(\tilde{x})|
	=&|\mathbb{E}f_t(r_1^n,r_2^n,K^{r_1^n,r_2^n}(t,0,x_n))-\mathbb{E}f_t(r_1,r_2,K^{r_1,r_2}(t,0,x))|\\
	\leq& |\mathbb{E}f_t(r_1^n,r_2^n,K^{r_1^n,r_2^n}(t,0,x_n))-\mathbb{E}f_t(r_1^n,r_2^n,K^{r_1,r_2}(t,0,x))|\\
	&+|\mathbb{E}f_t(r_1^n,r_2^n,K^{r_1,r_2}(t,0,x))-\mathbb{E}f_t(r_1,r_2,K^{r_1,r_2}(t,0,x))|\\
	=:&A_1^n+A_2^n.
	\end{split}
	\end{equation}
	Since $f\in C_b(\tilde{\mathbb X})$, then $f_t\in C_b(\tilde{\mathbb X})$ and $f_t(r_1^n,r_2^n,K^{r_1,r_2}(t,0,x))\xrightarrow{a.s.} f_t(r_1,r_2,K^{r_1,r_2}(t,0,x))$ as $n\rightarrow \infty$. By Lebesgue's 
	dominated  convergence theorem, we have
	\begin{equation}
	\begin{split}
	A_2^n=|\mathbb{E}f_t(r_1^n,r_2^n,K^{r_1,r_2}(t,0,x))-\mathbb{E}f_t(r_1,r_2,K^{r_1,r_2}(t,0,x))| \xrightarrow{n\rightarrow \infty} 0.
	\end{split}
	\end{equation}
	Furthermore, let $$b_n=\sup_{t\in \mathbb{R}, x\in \mathbb{R}^d}|\tilde{b}^{r_1^n,r_2^n}(t,x)-\tilde{b}^{r_1,r_2}(t,x)|$$ and 
	$$\sigma_n=\sup_{t\in \mathbb{R}, x\in \mathbb{R}^d}|\tilde{\sigma}^{r_1^n,r_2^n}(t,x)-\tilde{\sigma
	}^{r_1,r_2}(t,x)|.$$
    By (4) in Condition \ref{Quasi-dissipative}, we know that $b_n+\sigma_n\leq C(|r_1-r_1^n|^{\gamma}+|r_2-r_2^n|^{\gamma})$. Then $\lim_{n\rightarrow \infty}b_n=\lim_{n\rightarrow \infty}\sigma_n=0$. Similar to the estimation of \eqref{2020jun4}, we know that 
    $$E\big[\big|K^{r_1^n,r_2^n}(t,0,x_n)-K^{r_1,r_2}(t,0,x)\big|^2\big]\leq C(\alpha,\beta)\big(|x_n-x|^2+|r_1^n-r_1|^{2\gamma}+|r_2^n-r_2|^{2\gamma}\big).$$
   Then we have $K^{r_1^n,r_2^n}(t,0,x_n)\xrightarrow[n\rightarrow \infty]{L^2} K^{r_1,r_2}(t,0,x)$. Let 
	$$R_N=\{\omega: |K^{r_1,r_2}(t,0,x,\omega)|\leq N\}$$
	and
	$$R_N^n=\{\omega: |K^{r_1^n,r_2^n}(t,0,x_n,\omega)|\leq N\}.$$
	Then by the Chebyshev inequality we have $\lim_{N\rightarrow \infty}(\inf_{n\in \mathbb{N}}P(R_N^n\cap R_N))=1$. Since $f$ is continuous, then it is uniformly continuous on all compact subset of $\tilde{\mathbb X}$. Then for arbitrary $\epsilon>0$, there exists $\delta_N^{\epsilon}>0$ such that when $(t_1,t_2,x), (s_1,s_2,y)\in [0,\tau_1)\times [0,\tau_2) \times \bar{B}_N(0)$, where $\bar{B}_{N}(0)$ is a closed ball centred at 0 with radius $N$ in $\mathbb{R}^d$, and $d_1(t_1,s_1)+d_2(t_2,s_2)+|x-y|<\delta_N^{\epsilon}$, we have $|f((t_1,t_2,x))-f((s_1,s_2,y))|<\epsilon$. Set
	$$C^n_{\delta_N^{\epsilon}}=\{\omega: |K^{r_1^n,r_2^n}(t,0,x_n)-K^{r_1,r_2}(t,0,x)|<\delta_N^{\epsilon}\}.$$
	Then also by the Chebyshev inequality $\lim_{n\rightarrow \infty}P(C^n_{\delta_N^{\epsilon}})=1$. Hence for all $\omega\in C^n_{\delta_N^{\epsilon}}\cap R_N^n\cap R_N$,
	$$|f_t(r_1^n,r_2^n,K^{r_1^n,r_2^n}(t,0,x_n))-f_t(r_1^n,r_2^n,K^{r_1,r_2}(t,0,x))|<\epsilon.$$
	Therefore
	\begin{equation}
	\begin{split}
	\limsup_{n\rightarrow \infty}A_1^n
	=&\limsup_{n\rightarrow \infty} |\mathbb{E}f_t(r_1^n,r_2^n,K^{r_1^n,r_2^n}(t,0,x_n))-\mathbb{E}f_t(r_1^n,r_2^n,K^{r_1,r_2}(t,0,x))|\\
	\leq&\epsilon+2\|f\|_{\infty}\limsup_{n\rightarrow \infty} [(1-P(C^n_{\delta_N^{\epsilon}}))+(1-P(R_N^n\cap R_N))]\\
	=&\epsilon.
	\end{split}
	\end{equation}
	Since $\epsilon>0$ is arbitrary, we have $A_1^n\xrightarrow{n\rightarrow \infty} 0$. We complete the proof of $\tilde{P}_tf(\tilde{x}_n)\xrightarrow{n\rightarrow \infty} \tilde{P}_tf(\tilde{x})$.
	
\end{proof}

From Lemma \ref{Lemma of tight measure set} and Proposition \ref{Feller property of P^*}, we have the existence of invariant measure under $\tilde{P}^*$.

\begin{theorem}
	\label{Invariant measure}
	Assume Conditions \ref{Quasi-periodic condition}, \ref{Quasi-dissipative} and $\alpha>\frac{\beta^2}{2}$. Then there exists a uniquen invariant probability measure with respect to the semi-group $\tilde{P}^*$ which is given by $$\frac{1}{\tau_1\tau_2}\int_0^{\tau_1}\int_0^{\tau_2}\delta_{s_1}\times\delta_{s_2}\times\tilde{\rho}_{s_1,s_2}ds_1ds_2.$$
	Moreover, this invariant measure is ergodic with respect to the semigroup $\tilde{P}^*$.
\end{theorem}
\begin{proof}
	Existence: From Lemma \ref{Lemma of tight measure set}, we know that $\mathcal{M}$ defined by (\ref{Tight measure set}) is tight and hence weakly compact. This means that there exists a sequence $\{T_n\}_{n\geq 1}$ with $T_n\uparrow \infty$ as $n\rightarrow \infty$ and a probability measure $\bar{\tilde{\mu}}\in \mathcal{P}(\tilde{\mathbb X})$ such that $\bar{\tilde{\mu}}_{T_n}\xrightarrow{w} \bar{\tilde{\mu}}$. Moreover, for any fixed $t>0$, since 
	\begin{equation}
		\begin{split}
		\tilde{P}^*_t\bar{\tilde{\mu}}_{T_n}-\bar{\tilde{\mu}}_{T_n}
		&=\frac{1}{T_n}\int_{0}^{T_n}\tilde{P}^*_t\tilde{\mu}_sds-\frac{1}{T_n}\int_{0}^{T_n}\tilde{\mu}_sds\\
		&=\frac{1}{T_n}\int_{0}^{T_n}\tilde{\mu}_{t+s}ds-\frac{1}{T_n}\int_{0}^{T_n}\tilde{\mu}_sds\\
		&=\frac{1}{T_n}\int_{t}^{t+T_n}\tilde{\mu}_sds-\frac{1}{T_n}\int_{0}^{T_n}\tilde{\mu}_sds\\
		&=\frac{1}{T_n}\int_{T_n}^{t+T_n}\tilde{\mu}_sds-\frac{1}{T_n}\int_{0}^{t}\tilde{\mu}_sds,\\
		\end{split}
	\end{equation}
	so
	\begin{equation*}
		\begin{split}
		\limsup_{n\rightarrow \infty}\|\tilde{P}^*_t\bar{\tilde{\mu}}_{T_n}-\bar{\tilde{\mu}}_{T_n}\|_{BV}
		 &\leq \limsup_{n\rightarrow \infty}\frac{1}{T_n}(\int_0^t\|\tilde{\mu}_s\|_{BV}ds+\int_{T_n}^{T_n+t}\|\tilde{\mu}_s\|_{BV}ds)\\
		 &\leq \limsup_{n\rightarrow \infty}\frac{2t}{T_n}=0.
		\end{split}
	\end{equation*}
	Hence $\tilde{P}^*_t\bar{\tilde{\mu}}_{T_n}\xrightarrow{w} \bar{\tilde{\mu}}$. On the other hand, for any $f\in C_b(\tilde{\mathbb X})$, by Proposition \ref{Feller property of P^*}, we have $\tilde{P}_tf\in C_b(\tilde{\mathbb X})$, and therefore
	\begin{equation}
		\begin{split}
		\lim_{n\rightarrow \infty}\int_{\tilde{\mathbb X}}f(\tilde{y})\tilde{P}^*_t\bar{\tilde{\mu}}_{T_n}(d\tilde{y})
		 &=\lim_{n\rightarrow \infty}\int_{\tilde{\mathbb X}} \int_{\tilde{\mathbb X}}f(\tilde{y})\tilde{P}(t, \tilde{x}, d\tilde{y}) \bar{\tilde{\mu}}_{T_n}(d\tilde{x})\\
		 &=\lim_{n\rightarrow \infty}\int_{\tilde{\mathbb X}}\tilde{P}_tf(\tilde{x})\bar{\tilde{\mu}}_{T_n}(d\tilde{x})\\
		 &=\int_{\tilde{\mathbb X}}\tilde{P}_tf(\tilde{x})\bar{\tilde{\mu}}(d\tilde{x})\\
		 &=\int_{\tilde{\mathbb X}}f(\tilde{y})\tilde{P}^*_t\bar{\tilde{\mu}}(d\tilde{y}).
		\end{split}
	\end{equation}
    This means $\tilde{P}^*_t\bar{\tilde{\mu}}_{T_n}\xrightarrow{w} \tilde{P}^*_t\bar{\tilde{\mu}}$.  Summarising above we have that $\tilde{P}^*_t\bar{\tilde{\mu}}=\bar{\tilde{\mu}}$.
    
    Moreover, by \eqref{2020jun1*} in Lemma \ref{Continuous property of K}, we know that 
    $$\lim_{(t,s)\rightarrow (t_0,s_0)}\|\tilde{\varphi}(t,s)-\tilde{\varphi}(t_0,s_0)\|_2^2=0.$$   
    Then similar to the proof of Proposition \ref{Weakly convergence of P(t,s)}, we know that $\tilde{\rho}$ is continuous under the weak topology in $\mathcal{P}(\mathbb{R}^d)$, i.e. for all $f\in C_b(\mathbb{R}^d)$,
    $$\lim_{(t,s)\rightarrow (t_0,s_0)}\int_{\mathbb{R}^d}f(x)\tilde{\rho}_{t,s}(dx)=\int_{\mathbb{R}^d}f(x)\tilde{\rho}_{t_0,s_0}(dx).$$
    Let $\hat{\mu}$ defined by (\ref{Quasi-periodic measure after lift}). It is easy to check that $\hat{\mu}$ is also continuous under the weak topology in $\mathcal{P}(\tilde{\mathbb X})$.
    Since $\frac{1}{\tau_1}$ and $\frac{1}{\tau_2}$ are rationally linearly independent, by definition 5.1 in \cite{P.Walters1982}, $T_t: [0,\tau_1)\times [0,\tau_2) \rightarrow [0,\tau_1)\times [0,\tau_2)$ defined by 
    $$T_t(s_1,s_2)=(t+s_1\mod \tau_1,\ t+s_2\mod \tau_2), \text{ for all } s_1,s_2\in [0,\tau_1)\times [0,\tau_2)$$
    is a minimal rotation. Then applying Theorem 6.20 in \cite{P.Walters1982}, we know that $\frac{1}{\tau_1\tau_2}L$ is a unique ergodic probability measure on $[0,\tau_1)\times [0,\tau_2)$, where $L$ present the Lebesgue measures. Hence by Birkhoff's ergodic theory,
    \begin{equation*}
    \begin{split}
    \bar{\tilde{\mu}}_T&=\frac{1}{T}\int_{0}^{T}\tilde{\mu}_tdt\\
           &=\frac{1}{T}\int_{0}^{T}\hat{\mu}_{T_t(0,0)}dt\\
           &\xrightarrow[T\rightarrow \infty]{w}\int_{[0, \tau_1) \times [0, \tau_2)}\hat{\mu}_{s_1,s_2}\frac{1}{\tau_1\tau_2}ds_1ds_2.
    \end{split}
    \end{equation*}
    So
    $$\bar{\tilde{\mu}}=\int_{[0, \tau_1) \times [0, \tau_2)}\hat{\mu}_{s_1,s_2}\frac{1}{\tau_1\tau_2}ds_1ds_2=\frac{1}{\tau_1\tau_2}\int_0^{\tau_1}\int_0^{\tau_2}\delta_{s_1}\times\delta_{s_2}\times\tilde{\rho}_{s_1,s_2}ds_1ds_2$$
    is an invariant measure with respect to $\tilde{P}^*.$
    
    Uniqueness: We need to prove that for any invariant probability measure $\upsilon$, we have $\upsilon=\bar{\tilde{\mu}}$. By Lemma \ref{Two measure are the same}, we only need to prove that for any open set $\tilde{\mathcal{O}}\in \mathcal{B}(\tilde{\mathbb{X}})$, we have $\upsilon(\tilde{\mathcal{O}})\geq \bar{\tilde{\mu}}(\tilde{\mathcal{O}})$. Define
    $$\tilde{\mathcal{O}}^{r_1,r_2}=\{x\in \mathbb{R}^d: (r_1\mod \tau_1, \ r_2\mod \tau_2, x)\in \tilde{\mathcal{O}}\},$$
    $$\tilde{\mathcal{O}}^{r_1,r_2}_{\delta}=\{x: dist(x, (\tilde{\mathcal{O}}^{r_1,r_2})^c)>\delta\},$$
    and
    $$\tilde{\mathcal{O}}^{\delta}=\bigcup_{(s_1,s_2)\in [0,\tau_1)\times [0,\tau_2)}(s_1,s_2)\times\tilde{\mathcal{O}}^{s_1,s_2}_{\delta}.$$
    We know that $\tilde{\mathcal{O}}^{r_1,r_2}, \tilde{\mathcal{O}}^{r_1,r_2}_{\delta}$ and $\tilde{\mathcal{O}}^{\delta}$ are open sets, $\tilde{\mathcal{O}}^{r_1,r_2}_{\delta}\uparrow \tilde{\mathcal{O}}^{r_1,r_2}$ and $\tilde{\mathcal{O}}^{\delta}\uparrow \tilde{\mathcal{O}}$ as $\delta \downarrow 0$. Then
    \begin{equation}
    \label{Equation of invariant measure v}
    	\begin{split}
    	\upsilon\left(\tilde{\mathcal{O}}\right)
    	=&\lim_{T\rightarrow \infty}\frac{1}{T}\int_0^T\tilde{P}^*_t\upsilon\left(\tilde{\mathcal{O}}\right)dt\\
    	=&\lim_{T\rightarrow \infty}\frac{1}{T}\int_0^T\int_{\tilde{\mathbb X}}\tilde{P}\left(t,(s_1,s_2,x),\tilde{\mathcal{O}}\right)\upsilon(d\tilde{x})dt\\
    	=&\lim_{T\rightarrow \infty}\int_{\tilde{\mathbb X}} \frac{1}{T}\int_0^T P\left(K^{s_1,s_2}(t,0,x)\in \tilde{\mathcal{O}}^{t+s_1,t+s_2}\right)dt\upsilon\left(d\tilde{x}\right).
    	\end{split}
    \end{equation}
    Applying Remark \ref{Solution K shift time in r} and measure preserving transformation $\theta_t$, it follows that 
    $$\upsilon\left(\tilde{\mathcal{O}}\right)=\lim_{T\rightarrow \infty}\int_{\tilde{\mathbb X}} \frac{1}{T}\int_0^TP\left(K^{t+s_1,t+s_2}(0,-t,x)\in \tilde{\mathcal{O}}^{t+s_1,t+s_2}\right)dt\upsilon(d\tilde{x}).$$
    Similar to the proof of Theorem \ref{Existence and uniqueness of random path}, Lemma \ref{Bounded solution} and Lemma \ref{Exponential decay}, it can be shown that 
    the solution $K^{r_1,r_2}$ of (\ref{Solution K_r_1,r_2}) has the following estimate
    $$\|K^{r_1,r_2}(t,s,x)-\tilde{\varphi}^{r_1,r_2}(t)\|_2\leq Ce^{-(\alpha-\beta^2/2)(t-s)},$$
    for all $r_1,r_2\in\mathbb{R}, t\geq s$, where $C=C(\alpha,\beta,\tilde{M})$ only depends on $\alpha,\beta,\tilde{M}$ with $\tilde{M}=\sup_{t,s\in \mathbb{R}}(|\tilde{b}(t,s,0)|+\|\tilde{\sigma}(t,s,0)\|).$ Then for all $\delta>0$, by the Chebyshev inequality, we have
    \begin{equation}
    \label{Equation invariant measure v part 2}
    \begin{split}
    &P\left(K^{t+s_1,t+s_2}(0,-t,x)\in \tilde{\mathcal{O}}^{t+s_1,t+s_2}\right)\\
    \geq&P\left(\tilde{\varphi}^{t+s_1,t+s_2}(0)\in \tilde{\mathcal{O}}^{t+s_1,t+s_2}_{\delta}, \  |K^{t+s_1,t+s_2}(0,-t,x)-\tilde{\varphi}^{t+s_1,t+s_2}(0)|<\delta\right)\\
    \geq&P\left(\tilde{\varphi}^{t+s_1,t+s_2}(0)\in \tilde{\mathcal{O}}^{t+s_1,t+s_2}_{\delta}\right)-P\left(|K^{t+s_1,t+s_2}(0,-t,x)-\tilde{\varphi}^{t+s_1,t+s_2}(0)|\geq\delta\right)\\
    \geq&\tilde{\rho}_{t+s_1,t+s_2}\left(\tilde{\mathcal{O}}^{t+s_1,t+s_2}_{\delta}\right)-\frac{C^2}{\delta^2}e^{-2(\alpha-\beta^2/2)t}\\
    =&\hat{\mu}_{t+s_1,t+s_2}\left(\tilde{\mathcal{O}}^{\delta}\right)-\frac{C^2}{\delta^2}e^{-2(\alpha-\beta^2/2)t}.\\
    \end{split}
    \end{equation}
    Thus it turns out from (\ref{Equation of invariant measure v}), (\ref{Equation invariant measure v part 2}) and Fatou's Lemma that
     \begin{equation}
     \begin{split}
     \upsilon\left(\tilde{\mathcal{O}}\right)
     \geq&\liminf_{T\rightarrow \infty}\int_{\tilde{\mathbb X}} \frac{1}{T}\int_0^T\left(\hat{\mu}_{t+s_1,t+s_2}\left(\tilde{\mathcal{O}}^{\delta}\right)-\frac{C^2}{\delta^2}e^{-2(\alpha-\beta^2/2)t}\right)dt\upsilon(d\tilde{x})\\
     \geq&\int_{\tilde{\mathbb X}}\left(\liminf_{T\rightarrow \infty} \frac{1}{T}\int_0^T\hat{\mu}_{t+s_1,t+s_2}\left(\tilde{\mathcal{O}}^{\delta}\right)dt-\lim_{T\rightarrow \infty}\frac{C^2}{2\delta^2(\alpha-\beta^2/2)T}\right)\upsilon(d\tilde{x})\\
     \geq&\int_{\tilde{\mathbb X}}\left(\liminf_{T\rightarrow \infty} \frac{1}{T}\int_0^T\hat{\mu}_{t+s_1,t+s_2}\left(\tilde{\mathcal{O}}^{\delta}\right)dt\right)\upsilon(d\tilde{x}).
     \end{split}
     \end{equation}
     Again by Birkhoff's ergodic theory, we know that for all $(s_1,s_2)\in \mathbb{R}^2$
     $$\frac{1}{T}\int_0^T\hat{\mu}_{t+s_1,t+s_2}dt\xrightarrow[T\rightarrow \infty]{w} \bar{\tilde{\mu}}.$$
     Then since $\mathcal{O}^{\delta}$ is open, and by Proposition 2.4 in \cite{N. Ikeda-S. Watanabe}, we have
     $$\upsilon\left(\tilde{\mathcal{O}}\right) \geq \bar{\tilde{\mu}}\left(\mathcal{O}^{\delta}\right).$$
     Since $\mathcal{O}^{\delta}\uparrow \mathcal{O}$ as $\delta\downarrow 0$, the desired result follows from the continuity of measures with respect to an increasing sequence of sets. 
	 
	Moreover, by Theorem 3.2.6 in \cite{Da Prato1}, we know that this unique invariant measure $\bar{\tilde{\mu}}$ is ergodic.
\end{proof}

\begin{remark}
	It is not obvious how to check directly that $\frac{1}{\tau_1\tau_2}\int_0^{\tau_1}\int_0^{\tau_2}\delta_{s_1}\times\delta_{s_2}\times\tilde{\rho}_{s_1,s_2}ds_1ds_2$ is an invariant measure with respect to $\tilde{P}^*$ without appealing to the tightness argument.
\end{remark}
By a similar proof of Lemma \ref{Lemma of tight measure set}, Proposition \ref{Feller property of P^*} and Theorem \ref{Invariant measure}, 
it is not difficult to derive a general theorem. Here we denote by $\mathbb{X}$ a metric space, $\mathcal{B}(\mathbb{X})$ the Borel $\sigma$-algebra on $\mathbb{X}$, $B_b(\mathbb{X})$ the linear space of all $\mathcal{B}(\mathbb{X})$-bounded measurable functions and $\mathcal{P}(\mathbb{X})$ the collection of all probability measures on $(\mathbb{X}, \mathcal{B}(\mathbb{X}))$. Assume that $P(t,x,\Gamma), t\geq 0, x\in\mathbb{X}, \Gamma \in \mathcal{B}(\mathbb{X})$, is a Markovian transition function on $\mathbb{X}$. Denote by $P_t, t\geq 0: B_b(\mathbb{X}) \rightarrow B_b(\mathbb{X})$ and $P^*_t, t\geq 0: \mathcal{P}(\mathbb{X}) \rightarrow \mathcal{P}(\mathbb{X})$,
 the Markovian semi-groups associated with $P(t,x,\cdot)$. We say $\rho: \mathbb{R}\rightarrow \mathcal{P}(\mathbb{X})$ is an entrance measure with respect to $P^*$ if $P^*_t\rho_s=\rho_{t+s}$ for all $t\in \mathbb{R}^+, s\in \mathbb{R}$.
 We say $\rho$ is quasi-periodic if exists a measure-valued function $\tilde{\rho}_{s_1,s_2}$ satisfying the same relation with $\rho_s$ as in Definition \ref{Definition of quasi-periodic measure}. However we do not have the uniqueness of invariant measure in the general case. 
 
\begin{theorem}
	Assume the entrance measure $\rho$ with respect to $P_t^*, t\geq 0$, is a quasi-periodic measure with periods $\tau_1$ and $\tau_2$, where the reciprocals of $\tau_1$ and $\tau_2$ are rationally linearly independent. If $\{\bar{\rho}_T=\frac{1}{T}\int_0^T\rho_sds: T\in \mathbb{R}^+ \}$ is tight and the Markovian semi-group $P_t, t\geq 0$, is Feller, then there exists one invariant measure given by 
	$$\frac{1}{\tau_1\tau_2}\int_0^{\tau_1}\int_0^{\tau_2}\tilde{\rho}_{s_1,s_2}ds_1ds_2.$$
\end{theorem}

  \section{Density of entrance measure and quasi-periodic measure}
  In this section, we will give a sufficient condition to guarantee the existence of the density of the entrance measure. We need an extra condition.
  
   \begin{condition}
   	\label{Sigma invertible and bounded}
  	The diffusion coefficient $\sigma$ in SDE (\ref{SDE}) is invertible with $\sup_{t\in \mathbb{R}, x\in \mathbb{R}^d}\|\sigma^{-1}(t,x)\|<\infty.$
  \end{condition}
  
  We now give the definition of the well-known BMO space and some lemmas which will used in this section.
  
  \begin{definition}
  	Denote by BMO(s,t) the space of all $(\mathcal{F}_s^r)_{s\leq r\leq t}$-adapted $\mathbb{R}^d$-valued process $M$ with
  	$$\|M\|_{BMO(s,t)}:=\sup_{T\in \mathcal{T}_s^t}\left\|\left(\mathbb{E}\left[\int_{T}^{t}|M_r|^2dr|\mathcal{F}_s^T\right]\right)^{\frac{1}{2}}\right\|_{L^{\infty}}<\infty,$$
  	where $s<t$ and $\mathcal{T}_s^t$ is the set of stopping times taking their values in $[s,t]$.
  \end{definition}
  
  Then we have the following lemma.
  
  \begin{lemma}
  	\label{BMO implies Lp}
  	Let $M\in BMO(s,t)$. Then there exists $p>1$ such that 
  	$$\mathbb{E}\left[\left(\mathcal{E}\left(\int_{s}^{t}M_rdW_r\right)\right)^p\right]<\infty,$$
  	where $\mathcal{E}\left(\int_{s}^{t}M_rdW_r\right):=\exp\{\int_{s}^{t}M_rdW_r-\frac{1}{2}\int_{s}^{t}|M_r|^2dr\}$.
  \end{lemma}

  \begin{proof}
  	By Theorem 3.1 in \cite{Kazamaki94}, we know that if $\|M\|_{BMO(s,t)}\leq \Phi(p)$ for some $p>1$, where $\Phi$ is a continuous monotone function from $(1,\infty)$ to $\mathbb{R}_+$ with $\Phi(1+)=\infty$ and $\Phi(\infty)=0$, then $\mathcal{E}\left(\int_{s}^{t}M_rdW_r\right)$ is in $L^p$.
  \end{proof}
  
  We also need the following lemma which is almost the same as Lemma 4.1 in \cite{FZZ19}.
  
  \begin{lemma}
  	\label{Equivalent law}
  	Assume Conditions \ref{Dissipative} and \ref{Sigma invertible and bounded} hold. Let $X_t^{s,x}$ be the solution of SDE (\ref{SDE}) and $Z_t^{s,x}$ be the solution of the following SDE
  	\begin{equation}
  	\label{Diffusion SDE}
  	\begin{cases}
  	dZ_t=\sigma(t, Z_t)dW_t,  \quad t\geq s,\\
  	Z_s=x\in \mathbb{R}^{d}.
  	\end{cases}
  	\end{equation}
  	Then the laws of $X_t^{s,x}$ and $Z_t^{s,x}$ are equivalent, i.e.
  	$$P^{X_t^{s,x}}(B)=\tilde{P}^{Z_t^{s,x}}(B), \text{ for all } B\in \mathcal{B}(\mathbb{R}^d),$$
  	where $\frac{d\tilde{P}}{dP}=\mathcal{E}\left(\int_{s}^{t}\sigma^{-1}(r,Z_r^{s,x})b(r,Z_r^{s,x})dW_r\right)$
  \end{lemma}

 \begin{proof}
 	This lemma can be proved by almost the same proof as them of Lemma 4.1 in \cite{FZZ19}. 
 \end{proof}

  Now we have the following theorem.
  
  \begin{theorem}
  	\label{Existence of density}
  	Assume Conditions \ref{Dissipative} and \ref{Sigma invertible and bounded} hold. If $\alpha>\frac{\beta^2}{2}$, then $P(t,s,x,\cdot)$ and the entrance measure $\rho_t$ are absolutely continuous with respect to the Lebesgue measure $L$ on $(\mathbb{R}^d, \mathcal{B}(\mathbb{R}^d))$, and hence have the densities $p(t,s,x,y)$ and $q(t,y)$ respectively.
  \end{theorem}
  
  \begin{proof}
  	First we prove that $P(t,s,x,\cdot)$ is absolutely continuous with respect to $L$, i.e. for any $\Gamma\in \mathcal{B}(\mathbb{R}^d)$, $L(\Gamma)=0$ implies $P(t,s,x,\Gamma)=P(X_t^{s,x}\in \Gamma)=0$. By Lemma \ref{Equivalent law}, we know that
  	\begin{equation}
  	\label{equation of Girsonov}
  	\begin{split}
  	P(X_t^{s,x}\in \Gamma)&=\tilde{P}(Z_t^{s,x}\in \Gamma)
  	       =\mathbb{E}_{\tilde{P}}[1_{\Gamma}(Z_t^{s,x})]\\
           &=\mathbb{E}\left[\mathcal{E}\left(\int_{s}^{t}\sigma^{-1}(r,Z_r^{s,x})b(r,Z_r^{s,x})dW_r\right)1_{\Gamma}(Z_t^{s,x})\right],
  	\end{split}
  	\end{equation}
  	where $Z_t^{s,x}$ is the solution of SDE (\ref{Diffusion SDE}). Set $T_n:=\inf_{t\geq s}\{|Z_t^{s,x}|\geq n\}$. Since $\mathbb{E}_{\tilde{P}}[\sup_{r\in [s,t]}|Z_r^{s,x}|^2]<\infty$, then we have
  	$$\tilde{P}(T_n>t)=\tilde{P}(\sup_{r\in [s,t]}|Z_r^{s,x}|\leq n)\rightarrow 1 \text{ as } n\rightarrow \infty.$$
  	Thus
  	\begin{equation}
  	\begin{split}
  	P(X_t^{s,x}\in \Gamma)&=\mathbb{E}_{\tilde{P}}[1_{\Gamma}(Z_t^{s,x})]\\
  	&=\mathbb{E}_{\tilde{P}}[1_{\Gamma}(Z_t^{s,x})1_{[s,T_n]}(t)]+\mathbb{E}_{\tilde{P}}[1_{\Gamma}(Z_t^{s,x})1_{(T_n, \infty)}(t)]\\
  	&\leq \lim_{n\rightarrow \infty}[\mathbb{E}_{\tilde{P}}[1_{\Gamma}(Z_t^{s,x})1_{[s,T_n]}(t)]+\tilde{P}(T_n<t)]\\
  	&=\lim_{n\rightarrow \infty}\mathbb{E}_{\tilde{P}}[1_{\Gamma}(Z_t^{s,x})1_{[s,T_n]}(t)]\\
  	&=\lim_{n\rightarrow \infty}\mathbb{E}\left[1_{[s,T_n]}(t)\mathcal{E}\left(\int_{s}^{t}\sigma^{-1}(r,Z_r^{s,x})b(r,Z_r^{s,x})dW_r\right)1_{\Gamma}(Z_t^{s,x})\right].\\
  	\end{split}
  	\end{equation}
  	Since 
  	$$1_{[s,T_n]}(t)\mathcal{E}\left(\int_{s}^{t}\sigma^{-1}(r,Z_r^{s,x})b(r,Z_r^{s,x})dW_r\right) \leq \mathcal{E}\left(\int_{s}^{t}1_{[s,T_n]}(r)\sigma^{-1}(r,Z_r^{s,x})b(r,Z_r^{s,x})dW_r\right),$$
  	we have
  	\begin{equation}
  	\begin{split}
  	P(X_t^{s,x}\in \Gamma)
  	&\leq \liminf_{n\rightarrow \infty}\mathbb{E}\left[\mathcal{E}\left(\int_{s}^{t}1_{[s,T_n]}(r)\sigma^{-1}(r,Z_r^{s,x})b(r,Z_r^{s,x})dW_r\right)1_{\Gamma}(Z_t^{s,x})\right].
  	\end{split}
  	\end{equation}
  	We only need to prove that if $L(\Gamma)=0$, then for all $n$
  	$$\mathbb{E}\left[\mathcal{E}\left(\int_{s}^{t}1_{[s,T_n]}(r)\sigma^{-1}(r,Z_r^{s,x})b(r,Z_r^{s,x})dW_r\right)1_{\Gamma}(Z_t^{s,x})\right]=0.$$
  	Let $a_n(r)=1_{[s,T_n]}(r)\sigma^{-1}(r,Z_r^{s,x})b(r,Z_r^{s,x})$. By Condition \ref{Sigma invertible and bounded}, we know that there exists $C>0$ such that $\sup_{r\in \mathbb{R}}|a_n(r)|\leq C$. Then 
  	$$\sup_{T\in \mathcal{T}_s^t}\left\|\left(\mathbb{E}\left[\int_{T}^{t}|a_n(r)|^2dr|\mathcal{F}_s^T\right]\right)^{\frac{1}{2}}\right\|_{L^{\infty}}\leq C\sqrt{t-s},$$
  	which means $a_n \in BMO(s,t)$. By Lemma \ref{BMO implies Lp}, there exists $p>1$ such that $$\gamma_n:=\left(\mathbb{E}\left[\left(\mathcal{E}\left(\int_{s}^{t}a_n(r)dW_r\right)\right)^p\right]\right)^{\frac{1}{p}}<\infty.$$
  	Since $Z_t^{s,x}=x+\int_{s}^{t}\sigma(r,Z_r^{s,x})dW_r$, note that $\int_{s}^{t}\sigma(r,Z_r^{s,x})dW_r$ is in law a Brownian motion with time $\hat{\sigma}_t=\int_{s}^{t}\|\sigma(r,Z_r^{s,x})\|^2dr$, i.e. there exists a standard Brownian motion $\tilde{W}$ such that $\int_{s}^{t}\sigma(r,Z_r^{s,x})dW_r\stackrel{d}{=}\tilde{W}_{\hat{\sigma}_t}$. Also notice
  	$$\sqrt{d}=\|\sigma(t,x)\sigma^{-1}(t,x)\|\leq \|\sigma(t,x)\| \|\sigma^{-1}(t,x)\|,$$
  	thus
  	$$\|\sigma(t,x)\|\geq \frac{\sqrt{d}}{\|\sigma^{-1}(t,x)\|}\geq \frac{\sqrt{d}}{\sup_{t\in \mathbb{R}, x\in \mathbb{R}^d}\|\sigma^{-1}(t,x)\|}=:\underline{\sigma},$$
  	which suggests that $\hat{\sigma}_t\geq \underline{\sigma}(t-s)$. Using Proposition 6.17 in Chapter 2 in \cite{Karatzas91}, we have
  	\begin{equation}
  	\begin{split}
  	\mathbb{E}\left[1_{\Gamma}(Z_t^{s,x})\right]
  	        &=\mathbb{E}\left[1_{\Gamma}(x+\tilde{W}_{\hat{\sigma_t}})\right]\\
  		    &=\mathbb{E}\left[\mathbb{E}\left[1_{\Gamma}(x+\tilde{W}_{\hat{\sigma_t}})|\mathcal{F}_{\hat{\sigma}_t-\underline{\sigma}(t-s)}\right]\right]\\
  		    &=\mathbb{E}\left[\mathbb{E}\left[1_{\Gamma}(x+y+\tilde{W}_{\underline{\sigma}(t-s)})\right]\left|_{y=\tilde{W}_{\hat{\sigma}_t-\underline{\sigma}(t-s)}}\right.\right].
  	\end{split}
  	\end{equation}
  	Note
  	\begin{equation*}
  	\begin{split}
  	\mathbb{E}\left[1_{\Gamma}(x+y+\tilde{W}_{\underline{\sigma}(t-s)})\right]
  	   &=\frac{1}{(2\pi\underline{\sigma}(t-s))^{d/2}|\det\Sigma|^{1/2}}\int_{\mathbb{R}^d}1_{\Gamma}(x+y+z)e^{-(1/2\underline{\sigma}(t-s))|\Sigma^{-1/2} z|^2}dz\\
  	   &\leq \frac{1}{(2\pi\underline{\sigma}(t-s))^{d/2}|\det\Sigma|^{1/2}}L(\Gamma),
  	\end{split}
  	\end{equation*}
  	where $W_1\sim \mathcal{N}(0, \Sigma)$. Then
  	$$\mathbb{E}\left[1_{\Gamma}(Z_t^{s,x})\right]\leq \frac{1}{(2\pi\underline{\sigma}(t-s))^{d/2}|\det\Sigma|^{1/2}}L(\Gamma).$$
  	Let $q$ be the dual number of $p$. Then by Cauchy-Schwarz inequality, 
  	\begin{equation}
  	\begin{split}
  	\mathbb{E}\left[\mathcal{E}\left(\int_{s}^{t}1_{[s,T_n]}(r)\sigma^{-1}(r,Z_r^{s,x})b(r,Z_r^{s,x})dW_r\right)1_{\Gamma}(Z_t^{s,x})\right]
  	&\leq \gamma_n \{\mathbb{E}[1_{\Gamma}(Z_t^{s,x})]\}^{\frac{1}{q}}\\
  	&\leq C_n\cdot L(\Gamma)^{\frac{1}{q}},
  	\end{split}
  	\end{equation}
  	where $C_n=\gamma_n\cdot \left(\frac{1}{(2\pi\underline{\sigma}(t-s))^{d/2}|\det\Sigma|^{1/2}}\right)^{\frac{1}{q}}$. 
  	
  	So if $L(\Gamma)=0$, then $\mathbb{E}\left[\mathcal{E}\left(\int_{s}^{t}1_{[s,T_n]}(r)\sigma^{-1}(r,Z_r^{s,x})b(r,Z_r^{s,x})dW_r\right)1_{\Gamma}(Z_t^{s,x})\right]=0$, and hence $P(t,s,x,\Gamma)=P(X_t^{s,x}\in \Gamma)=0$. Thus $P(t,s,x,\cdot)$ is absolutely continuous with respect to the Lebesgue measure and by Radon-Nikodym theorem, the density of $P(t,s,x,\cdot)$ with respect to the Lebesgue measure exists.
    
    For the entrance measure $\rho_t$, since 
    \begin{equation}
      \label{Entrance measure of markov kernel}
    	\rho_t(\Gamma)=P^*(t,s)\rho_s(\Gamma)=\int_{\mathbb{R}^d}P(t,s,x,\Gamma)\rho_s(dx),
    \end{equation}
    then if $L(\Gamma)=0$, we have $\rho_t(\Gamma)=0$. This also suggests that $\rho_t$ is absolutely continuous with respect to $L$ and thus its density exists.
  \end{proof}
  
  We already know the conditions to guarantee the existence of the densities $p(t, s, x, y)$ and $q(t,y)$ of the two- parameter Markov transition kernel $P(t,s,x,\cdot)$ and entrance measure $\rho_t$ respectively. By Fubini theorem, we know that
  $$\rho_t(\Gamma)=\int_{\mathbb{R}^d}P(t,s,x,\Gamma)\rho_s(dx)=\int_{\Gamma}\int_{\mathbb{R}^d}p(t,s,x,y)\rho_s(dx)dy=\int_{\Gamma}\int_{\mathbb{R}^d}p(t,s,x,y)q(s,x)(dx)dy.$$
  Then it is obvious that
  \begin{equation}
  	\label{Equation of entrance measure density}
  	q(t,y)=\int_{\mathbb{R}^d}p(t,s,x,y)q(s,x)(dx).
  \end{equation}
  In addition, we give the following condition:
  \begin{condition}
  	\label{Smooth of coefficients}
  	Assume $b=(b_i)_{i=1}^d$, $\sigma=(\sigma_{ij})_{i,j=1}^d$ in SDE \eqref{SDE} satisfy the following conditions:
  	\begin{description}
  		\item[(1)] The functions $b(t,x), \sigma(t,x)$ are globally bounded and uniformly H\"{o}lder-continuous in $(t,x)$.
  		\item[(2)] The functions $b(t, \cdot)\in C^1(\mathbb{R}^d; \mathbb{R}^d)$, $\sigma(t,\cdot)\in C^2(\mathbb{R}^d; \mathbb{R}^{d\times d})$ such that $\partial_{x_i}b_i, \partial_{x_ix_j}^2\sigma_{ij}$ are bounded and H\"{o}lder-continuous.
  	\end{description}
  \end{condition}
  Then it is well-known that (see \cite{Friedman1964,Friedman1975,Karatzas91} for more details) under Condition \ref{Smooth of coefficients}, $p(\cdot,s,x,\cdot)\in C^{1,2}(\mathbb{R}\times \mathbb{R}^d)$ satisfies the following Fokker-Planck equation
  \begin{equation}
  	\label{Fokker-Planck equation}
  	\partial_tp(t,s,x,y)=\mathcal{L}^*(t)p(t,s,x,y), \ t\geq s,
  \end{equation}
  with initial conditions $p(s,s,x,y)=\delta_x(y)$, where $\mathcal{L}^*(t)p$ is the Fokker-Planck operator given by 
  \begin{equation}
  \label{Fokker-Planck operator}
  	\mathcal{L}^*(t)p=-\sum_{i=1}^d\partial_{x_i}(b_i(t,y)p)+\frac{1}{2}\sum_{i,j=1}^d\partial_{x_ix_j}^2 \left(\sigma\sigma^T_{ij}(t,y)p\right).
  \end{equation}
  
  Now we have the following theorem.
  \begin{theorem}
  	\label{Theorem Fokker-Planck}
  	Assume Conditions \ref{Dissipative}, \ref{Sigma invertible and bounded}, \ref{Smooth of coefficients} hold and $\alpha>\frac{\beta^2}{2}$. Let $q \in C_+^{1,2}(\mathbb{R}\times \mathbb{R}^d)\bigcap L^1(\mathbb{R}^d)$ with $\|q(t,\cdot)\|_{L^1(\mathbb{R}^d)}=1$ for all $t$, and define $\rho: \mathbb{R}\rightarrow \mathcal{P}(\mathbb{R}^d)$ by
  	$$\rho_t(\Gamma)=\int_{\Gamma}q(t,y)dy, \text{ for all } t\in \mathbb{R}.$$
  	Then $\rho$ is an entrance measure if and only if $q$ satisfies the infinite horizon Fokker-Planck equation problem:
  	\begin{equation}
  		\label{Equation of Fokker-Planck}
		\partial_tq=\mathcal{L}^*(t)q, \ t\geq s
	\end{equation}
	for any $s\in\mathbb{R}$, and the additional condition
	\begin{equation}
		\label{norm of q is 1}
		\|q(t,\cdot)\|_{L^1(\mathbb{R}^d)}=1, \ t\in \mathbb{R},
	\end{equation}
	and the solution to infinite horizon Fokker-Planck equation problem is unique. Hence the solution of (\ref{Equation of Fokker-Planck}) and the entrance measure have one to one correspondence.
  \end{theorem}

  \begin{proof}
  	Assume first that $\rho$ is an entrance measure. We already know that $p,q$ satisfy (\ref{Equation of entrance measure density}) and $p(t,s,x,y)$ satisfies Fokker-Planck equation (\ref{Fokker-Planck equation}). We take the derivative with respect to t on both sides of (\ref{Equation of entrance measure density}) to have
  	\begin{equation}
  	\begin{split}
  	\partial_tq(t,x)=&\int_{\mathbb{R}^d}\partial_tp(t,s,y,x)q(s,y)dy\\
  	                     =&\int_{\mathbb{R}^d}\mathcal{L}^*(t)p(t,s,y,x)q(s,y)dy\\
  	                     =&\int_{\mathbb{R}^d}\bigg(-\sum_{i=1}^d\partial_{x_i}(b_i(t,x)p(t,s,y,x))q(s,y)\bigg)dy\\
  	                       &+\int_{\mathbb{R}^d}\frac{1}{2}\sum_{i,j=1}^d\partial_{x_ix_j}^2 \left(\sigma\sigma^T_{ij}(t,x)p(t,s,y,x)\right)q(s,y)dy\\
  	                    =:& I+\uppercase\expandafter{\romannumeral2}.
  	\end{split}		
  	\end{equation}
  	For the first part, we have
  	\begin{equation}
  	\begin{split}
  	    I&=-\sum_{i=1}^d\int_{\mathbb{R}^d}[\partial_{x_i}(b_i(t,x))p(t,s,y,x) +b_i(t,x)\partial_{x_i}(p(t,s,y,x))]q(s,y)dy\\
  	     &=-\sum_{i=1}^d\partial_{x_i}(b_i(t,x))\int_{\mathbb{R}^d}p(t,s,y,x)q(s,y)dy -\sum_{i=1}^d b_i(t,x)\partial_{x_i}\int_{\mathbb{R}^d}p(t,s,y,x)q(s,y)dy\\
  	     &=-\sum_{i=1}^d\partial_{x_i}(b_i(t,x))q(t,x) -\sum_{i=1}^d b_i(t,x)\partial_{x_i}q(t,x)\\
  	     &=-\sum_{i=1}^d\partial_{x_i}(b_i(t,x)q(t,x)).
  	\end{split}		
  	\end{equation}
  	Similarly, for the second part, we have
  	$$II=\frac{1}{2}\sum_{i,j=1}^d\partial_{x_ix_j}^2 \left(\sigma\sigma^T_{ij}(t,x)q(t,x)\right).$$
  	Hence the density function $q(t,x)$ of entrance measure $\rho_t$ satisfies
  	$$\partial_tq=\mathcal{L}^*(t)q.$$
  	
  	Conversely, if $q$ is the solution of \eqref{Equation of Fokker-Planck} satisfying \eqref{norm of q is 1}. First from the heat kernel $p(t,s,x,y)$ of Fokker-Planck equation, we have \eqref{Equation of entrance measure density}. Then by Fubini's theorem, we have for all $\Gamma\in \mathcal{B}(\mathbb{R}^d)$, any $s\in\mathbb{R}$ and $t\geq s$
  	\begin{equation*}
  	\begin{split}
  		P^*(t,s)\rho_s(\Gamma)&=\int_{\mathbb{R}^d}P(t,s,y,\Gamma)\rho_s(dy)\\
  	        &=\int_{\mathbb{R}^d}\int_{\Gamma}p(t,s,y,x)dxq(s,y)dy\\
  	        &=\int_{\Gamma}\int_{\mathbb{R}^d}p(t,s,y,x)q(s,y)dydx\\
  	        &=\int_{\Gamma}q(t,x)dx\\
  	        &=\rho_t(\Gamma)
  	\end{split}
  	\end{equation*}
  	which means $\rho$ is an entrance measure. With the correspondence of the periodic measure and the solution of infinite horizon Fokker-Planck equation, since the entrance measure is unique in this theorem, we know that the solution of the infinite horizon Fokker-Planck equation problem is unique.
  \end{proof}
 
  Now assume that $u^r(t,s,x)$ and $K^{r_1,r_2}(t,s,x)$ are the solutions of equation (\ref{Solution u_r}) and (\ref{Solution K_r_1,r_2}) respectively, and the corresponding semi-groups $P^r, P^{r_1,r_2}$ defined as 
  \begin{equation}
  \begin{cases}
  P^r(t,s,x,\Gamma):=P(u^r(t,s,x)\in\Gamma)\\
  P^{r_1.r_2}(t,s,x,\Gamma):=P(K^{r_1,r_2}(t,s,x)\in\Gamma).
  \end{cases}
  \end{equation}
  We can also define $P^{r,*}(t,s) \ (resp. \ P^{r_1,r_2,*}(t,s))$ as in (\ref{Define measure transition P^*}) when we replace $\{P^*(t,s), P(t,s,x,\Gamma)\}$ by $\{P^{r,*}(t,s), P^r(t,s,x,\Gamma)\} \ (resp. \ \{P^{r_1,r_2,*}(t,s), P^{r_1.r_2}(t,s,x,\Gamma)\})$.
  Let $\varphi^r(t), \varphi^{r_1,r_2}(t)$ be defined as in (\ref{L2-lim K u}), and $\rho^r_t, \rho^{r_1,r_2}_t$ be the laws of $\varphi^r(t), \varphi^{r_1,r_2}(t)$ respectively. Then we have 
  $$P^{r,*}(t,s)\rho^r_s=\rho^r_t, \ P^{r_1,r_2,*}(t,s)\rho^{r_1,r_2}_s=\rho^{r_1,r_2}_t.$$
  Similar to Condition \ref{Sigma invertible and bounded} and \ref{Smooth of coefficients}, we give the following condition.  
  \begin{condition}
  	\label{Quasi sigma invertible and bounded}
  	The functions $\tilde{b}=(\tilde{b}_i)_{i=1}^d, \tilde{\sigma}=(\tilde{\sigma}_{ij})_{i,j=1}^d$ in Condition \ref{Quasi-periodic condition} satisfy the following conditions: 
  	\begin{description}
  		\item[(1)] The functions $\tilde{b}(t,s,x), \tilde{\sigma}(t,s,x)$ are globally bounded and uniformly H\"{o}lder-continuous in $(t,s,x)$.
  		\item[(2)] The functions $\tilde{b}(t, s, \cdot)\in C^1(\mathbb{R}^d; \mathbb{R}^d)$, $\tilde{\sigma}(t,s,\cdot)\in C^2(\mathbb{R}^d; \mathbb{R}^{d\times d})$ such that $\partial_{x_i}\tilde{b}_i, \partial_{x_ix_j}^2\tilde{\sigma}_{ij}$ are bounded and H\"{o}lder-continuous.
  		\item[(3)] The function $\tilde{\sigma}(t,s,x)$ is invertible with $\sup_{t,s\in \mathbb{R}, x\in \mathbb{R}^d}\|\tilde{\sigma}^{-1}(t,s,x)\|<\infty$.
  	\end{description}
  \end{condition}
  Then by Theorem \ref{Existence and uniqueness of entrance measure} and Theorem \ref{Existence of density}, we can directly deduce the following theorem
  \begin{theorem}
  	Assume Conditions \ref{Quasi-periodic condition}, \ref{Quasi-dissipative} and \ref{Quasi sigma invertible and bounded} hold. If $\alpha>\frac{\beta^2}{2}$, then $\rho^r, \rho^{r_1,r_2}$ are the entrance measures of equation (\ref{Solution u_r}) and (\ref{Solution K_r_1,r_2}) respectively. Moreover $P^r(t,s,x,\cdot), P^{r_1,r_2}(t,s,x,\cdot)$ and the entrance measures $\rho^r_t, \rho^{r_1,r_2}_t$ are absolutely continuous with respect to the Lebesgue measure $L$ on $(\mathbb{R}^d, \mathcal{B}(\mathbb{R}^d))$, and hence have the density $p^r(t,s,x,y)$, $p^{r_1,r_2}(t,s,x,y)$, $q^r(t,y)$, $q^{r_1,r_2}(t,y)$ respectively.
  \end{theorem}
  Similarly, we know that 
  $$q^r(t,x)=\int_{\mathbb{R}^d}p^r(t,s,y,x)q^r(s,y)(dy)$$
  and
  $$q^{r_1,r_2}(t,x)=\int_{\mathbb{R}^d}p^{r_1,r_2}(t,s,y,x)q^{r_1,r_2}(s,y)(dy).$$
  Moreover, $q^{r_1,r_2}$ (resp. $q^r$) satisfies the following quasi-periodic Fokker-Planck equation problem:
  $$\partial_tq^{r_1,r_2}=\mathcal{L}^{r_1,r_2,*}(t)q^{r_1,r_2}, t\geq s \ \ (resp.\ \partial_tq^r=\mathcal{L}^{r,*}(t)q^r, \ t\geq s)$$
  where $\mathcal{L}^{r_1,r_2,*}$ (resp. $\mathcal{L}^{r,*}$) is given in (\ref{Fokker-Planck operator}) where $b,\sigma$ is replaced by $\tilde{b}^{r_1,r_2}, \tilde{\sigma}^{r_1,r_2}$ (resp. $\tilde{b}^r, \tilde{\sigma}^r$).
  
  By the proof of Theorem \ref{Existence and uniqueness of quasi-periodic random path}, we know that $u^r(t,s,x,\cdot)=u(t+r,s+r,x,\theta_{-r}\cdot)$ and $\varphi^{r}(t,\cdot)=\varphi(t+r,\theta_{-r}\cdot)$. Since $\theta_{-r}$ preserves the probability measure $P$, then $P^r(t,s,x,\cdot)=P(t+r,s+r,x,\cdot)$ and $\rho^r_t=\rho_{t+r}$. Hence their densities have the following relations
  $$p^r(t,s,x,y)=p(t+r,s+r,x,y), \quad q^r(t,x)=q(t+r,x).$$
  \section*{Acknowledgements}

We are grateful to the anonymous referee for their constructive comments which lead to significant improvements of this paper. 
  We would like to thank Kening Lu and Hans Crauel for raising our interests to consider random quasi-periodicity in various occasions. 
    We acknowledge the financial support of a Royal Society Newton Fund grant (Ref NA150344) and an EPSRC Established Career Fellowship to HZ (Ref EP/S005293/2).

\end{document}